\documentclass[12pt]{amsart}

\usepackage[hidelinks]{hyperref} 

 \usepackage{amsaddr}
\usepackage{amsmath}
\usepackage{amsfonts}
\usepackage{amssymb}
\usepackage{latexsym}
\usepackage{graphicx}
\usepackage{enumitem}
\usepackage{multirow}
\usepackage{pgfplots}
\usepackage{subcaption}

\numberwithin{equation}{section}
\newtheorem{theorem}{Theorem}[section]

\newtheorem{lemma}[theorem]{Lemma}

\newtheorem{remark}[theorem]{Remark}
\newtheorem{definition}[theorem]{Definition}

\usepackage{mathtools}

\makeatletter
\DeclareRobustCommand\widecheck[1]{{\mathpalette\@widecheck{#1}}}
\def\@widecheck#1#2{%
\setbox\z@\hbox{\m@th$#1#2$}%
\setbox\tw@\hbox{\m@th$#1%
\widehat{%
\vrule\@width\z@\@height\ht\z@
\vrule\@height\z@\@width\wd\z@}$}%
\dp\tw@-\ht\z@
\@tempdima\ht\z@ \advance\@tempdima2\ht\tw@ \divide\@tempdima\thr@@
\setbox\tw@\hbox{%
\raise\@tempdima\hbox{\scalebox{1}[-1]{\lower\@tempdima\box
\tw@}}}%
{\ooalign{\box\tw@ \cr \box\z@}}}
\makeatother

\newcommand{\ep}{\epsilon}
\renewcommand{\hat}{\widehat}

\DeclareMathOperator{\sinc}{sinc}
\renewcommand{\tilde}{\widetilde}

\newcommand{\be}{\begin{equation}}
\newcommand{\ee}{\end{equation}}

\newcommand{\bes}{\begin{equation*}}
\newcommand{\ees}{\end{equation*}}

\newcommand{\bunderbrace}[2]{%
\begin{array}[t]{@{}c@{}}
\underbrace{#1}\\
#2
\end{array}
}

\newcommand{\mand}{\quad \text{and}\quad}

\renewcommand{\H}{{\mathcal{H}}}
\renewcommand{\L}{{\mathcal{L}}}
\newcommand{\N}{{\mathbf{N}}}
\renewcommand{\O}{{\mathcal{O}}}
\newcommand{\R}{{\mathbf{R}}}
\newcommand{\Z}{\mathbf{Z}}

\graphicspath{{Images/}}

\newcommand{\II}{I\!I}

\usepackage{tikz}
\usetikzlibrary{positioning}
\usetikzlibrary{patterns}
\usetikzlibrary{shapes}
\usetikzlibrary{matrix, arrows,decorations.pathmorphing}
\usetikzlibrary{decorations}
\usetikzlibrary{decorations.markings}
\usetikzlibrary{fadings}
\usetikzlibrary{arrows.meta,bending}
\usepackage{tikzscale}

\title{Solitary Waves in Mass-in-Mass Lattices}
\author{Timothy E. Faver}\address{Leiden University\\ \tt{t.e.faver@math.leidenuniv.nl}}
\author{Roy H. Goodman}\address{NJIT \\ \tt goodman@njit.edu}
\author{J. Douglas Wright}\address{Drexel University\\ {\tt{jdw66@drexel.edu}}\\ (corresponding author)}

\begin{document}
\maketitle
\begin{abstract} We consider the existence of spatially localized traveling wave solutions of the mass-in-mass lattice. Under an anti-resonance condition first
discovered by Kevrekidis, Stefanov and Xu, we prove that such solutions exist in two distinguished limits, the first where the mass of the internal resonator is small and the second where the internal spring is very stiff. We then numerically simulate the solutions and these simulations indicate that the anti-resonant traveling waves are {\it very} weakly unstable.

\end{abstract}

\section{Introduction}

\subsection{The traveling wave problem in mass-in-mass lattices}\label{tw problem intro}
We study the existence of solitary waves in mass-in-mass (MiM) lattices.
An MiM lattice consists of an infinite chain of particles, called beads, each coupled to its two nearest neighbors, and also coupled to an additional resonator particle.
See Figure~\ref{fig-mim} for a sketch of the particular MiM model that we will consider.

MiM lattices belong to a broad class of artificially constructed materials called granular metamaterials, which are prized in experimental settings for their simple structure and highly tunable properties; one can vary the masses of the beads and/or the resonators and adjust both the spring forces governing the beads' nearest-neighbor interaction as well as the bead-resonator coupling, thereby allowing a range of interaction strengths from the  purely linear to the highly nonlinear~\cite{nesterenko}.
Physically-constructed MiM lattices have served in a diverse array of experiments, such as constructing sensors for bone elasticity~\cite{bones} and ultrasonic scans~\cite{spadoni}, determining of the setting time for cement~\cite{cement}, modeling switches and logic gates~\cite{li}, and modeling vibration absorption in composite materials~\cite{bonanomi-et-al, gantzounis}.

The MiM lattices in this paper have the following form.
All of the beads are identical and are normalized to have mass 1.
We index the beads by integers $j \in \Z$.
The springs connecting the beads are also identical and all have the same potential $V \in C^1(\R)$; we will add some further hypotheses on $V$ later.
Likewise, the resonators are identical and have mass $\mu > 0$.
The bead-resonator spring is linear and exerts the force $\kappa{r}$ when stretched a distance $r$; here $\kappa > 0$ is fixed.
If we denote by $U_j$ the position of the center of mass of the $j$th bead and $u_j$ the position of the center of mass of the $j$th bead's resonator, then Newton's second law requires $U_j$ and $u_j$ to satisfy the following system:
\be\label{UrMiM}
\begin{cases}
\ddot{U}_j = V'(U_{j+1}-U_j)-V'(U_j-U_{j-1})+\kappa(u_j-U_j) \\
\\
\mu\ddot{u}_j = \kappa(U_j-u_j).
\end{cases}
\ee

\begin{remark}
Without loss of generality, we may assume $V'(0) = 0$, for if not we can put $V(r) = V'(0)r + W(r)$ with $W(r)$ implicitly defined. 
Making this change in~\eqref{UrMiM} simply replaces each $V$ with $W$ and we have $W'(0) = 0$.
\end{remark}

We work, as is typical in lattice models, in the relative displacement coordinates
\[
R_j := U_{j+1} - U_j
\mand
r_j := u_j-U_j.
\]
Then we make the following traveling wave ansatz:
\begin{equation}\label{ansatz}
R_j(t) = \rho_1(j-ct+1/2)
\mand
r_j(t) = \rho_2(j-ct).
\end{equation}
Here $\rho_1 = \rho_1(x)$ and $\rho_2 = \rho_2(x)$ are the traveling wave profiles and $c \in \R$ is the wave speed.
The shift by $1/2$ in $\rho_1$ induces a convenient symmetry we exploit at a later point. In any case,  $\rho_1$ and $\rho_2$ satisfy
\begin{subequations}\label{traveling}%
\be\label{i}
c^2 \rho_1'' - \delta^2 V'(\rho_1) - \kappa \delta \rho_2 = 0
\ee
\be\label{ii}
c^2 \mu \rho_2'' + \kappa(1+\mu) \rho_2 + \mu \delta V'(\rho_1)= 0.
\ee
\end{subequations}
Here $\delta$ is the linear difference operator
\begin{equation}\label{delta}
\delta f (x):=f(x+1/2)-f(x-1/2). 
\end{equation}

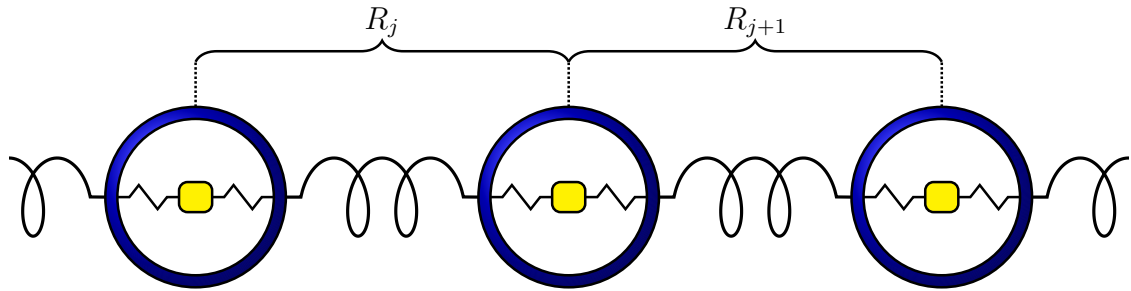
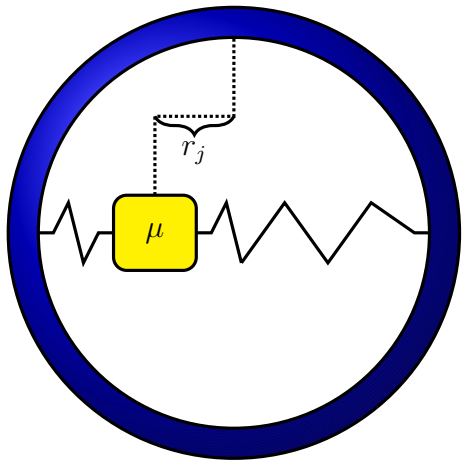
\begin{figure}\label{lattice fig}

\begin{subfigure}{1\textwidth}
\centering

\scalebox{.8}{
\begin{tikzpicture}[very thick]

\begin{scope}

\def\coilspan{2.7}; 
\def\h{.65}; 
\def\edgespan{.25}; 

\def\beadrad{1.5}; 
\def\innerad{1.3}; 

\def\M{.4}; 
\def\N{.55}; 
\def\E{.35}; 


\def\coil#1{ 
{\N+\M*\t+\E*sin(\t*pi r)+#1},
{\h*cos(\t*pi r)}
}


\def\NLSPRING#1{
\draw[line width=1.5pt,domain={-.5:4.5},smooth,variable=\t,samples=100]
(#1,0)--plot(\coil{#1+\edgespan})
--(#1+2*\edgespan+\coilspan,0);
}

\def\d{.2}; 


\def\BEAD#1{
\shade[ball color=blue,opacity=.8] (#1,0) circle(\beadrad);
\fill[white] (#1,0) circle(\innerad);
\draw[very thick] (#1,0) circle(\beadrad);
\draw[very thick] (#1,0) circle(\innerad);
}

\def\p{.15625}; 


\def\RESONATOR#1{
\draw
(#1+\d,0)
--(#1+2*\d,0)
--(#1+2*\d+\p,.2)
--(#1+2*\d+3*\p,-.2)
--(#1+2*\d+4*\p,0)
--(#1+2*\d+4*\p+.2,0);

\draw[rounded corners,fill=yellow] (#1+2*\d+4*\p+.2,.25) rectangle (#1+2*\d+4*\p+.2+.55,-.25);

\draw (#1+2*\d+4*\p+.2+.55,0)--(#1+2*\d+4*\p+.2+.55+.2,0) --(#1+2*\d+4*\p+.2+.55+.2+\p,.2)
--(#1+2*\d+4*\p+.2+.55+.2+3*\p,-.2)
--(#1+2*\d+4*\p+.2+.55+.2+4*\p,0)
--(#1+2*\d+4*\p+.2+.55+.2+4*\p+.2,0);
};

\clip (\beadrad+\edgespan+.5*\coilspan,4) rectangle (7*\beadrad+7*\edgespan+3.5*\coilspan,-2);

\BEAD{0};

\NLSPRING{\beadrad};
\BEAD{2*\beadrad+2*\edgespan+\coilspan};
\RESONATOR{1*\beadrad+2*\edgespan+\coilspan};
\draw[densely dotted] (2*\beadrad+2*\edgespan+\coilspan,\beadrad)--(2*\beadrad+2*\edgespan+\coilspan,\beadrad+.75);

\NLSPRING{3*\beadrad+2*\edgespan+\coilspan};
\BEAD{4*\beadrad+4*\edgespan+2*\coilspan};
\RESONATOR{3*\beadrad+4*\edgespan+2*\coilspan};
\draw[densely dotted] (4*\beadrad+4*\edgespan+2*\coilspan,\beadrad)--(4*\beadrad+4*\edgespan+2*\coilspan,\beadrad+.75);

\NLSPRING{5*\beadrad+4*\edgespan+2*\coilspan};
\BEAD{6*\beadrad+6*\edgespan+3*\coilspan};
\RESONATOR{5*\beadrad+6*\edgespan+3*\coilspan};
\draw[densely dotted] (6*\beadrad+6*\edgespan+3*\coilspan,\beadrad)--(6*\beadrad+6*\edgespan+3*\coilspan,\beadrad+.75);

\NLSPRING{7*\beadrad+6*\edgespan+3*\coilspan};

\draw[decoration={brace, amplitude = 10pt},decorate] (2*\beadrad+2*\edgespan+\coilspan,\beadrad+.75)--node[midway,above=7pt]{\scalebox{1.25}{$R_j$}}(4*\beadrad+4*\edgespan+2*\coilspan,\beadrad+.75);
\draw[decoration={brace, amplitude = 10pt},decorate] (4*\beadrad+4*\edgespan+2*\coilspan,\beadrad+.75)--node[midway,above=7pt]{\scalebox{1.25}{$R_{j+1}$}}(6*\beadrad+6*\edgespan+3*\coilspan,\beadrad+.75);

\end{scope}
\end{tikzpicture}
}

\caption{A snippet of the MiM lattice}
\end{subfigure}

\bigskip

\begin{subfigure}{1\textwidth}
\[
\begin{tikzpicture}

\shade[ball color=blue,opacity=.8]
(3,0) circle(3);
\fill[white] (3,0) circle(2.6);

\def\d{.2};
\def\h{.4}
\draw[very thick] (3,0) circle(3);
\draw[very thick] (3,0) circle(2.6);

\draw[very thick] 
(.4,0)
--(.4+\d,0)
--(.4+2*\d,\h)
--(.4+3*\d,-\h)
--(.4+4*\d,0)
--(.4+5*\d,0);

\fill[rounded corners=5, very thick,yellow] (.4+5*\d,.5) rectangle (.4+5*\d+1.1,-.5);
\node at (.4+5*\d+.55,0){$\mu$};
\draw[rounded corners=5, very thick] (.4+5*\d,.5) rectangle (.4+5*\d+1.1,-.5);

\draw[very thick,densely dotted] (.4+5*\d+.55,.5)--(.4+5*\d+.55,1.55)--(3,1.55)--(3,2.6);

\draw[decoration={brace, amplitude = 7pt,mirror},decorate,very thick] (.4+5*\d+.55,1.55)--node[midway,below=5pt]{$r_j$}(3,1.55);

\def\r{1.5};

\draw[very thick]
(\r+5*\d,0)
--(\r+6*\d,0)
--(\r+7*\d,\h)
--(\r+8*\d,-\h);

\def\s{.575};

\draw[very thick] (\r+8*\d,-\h)
--(\r+8*\d+\s,\h)
--(\r+8*\d+2*\s,-\h)
--(\r+8*\d+3*\s,\h)
--(\r+8*\d+4*\s,0)
--(\r+9*\d+4*\s,0);

\end{tikzpicture}
\]
\caption{A close-up of one bead-resonator pair}
\end{subfigure}

\caption{The mass-in-mass lattice}
\label{fig-mim}
\end{figure}

\subsection{Distinguished material limits}\label{dist}
We will look for solitary traveling wave solutions to~\eqref{traveling}.
That is, we seek solutions $\rho_1$ and $\rho_2$ that vanish exponentially fast at infinity in the sense that $\cosh(\cdot)^b\rho_j \in H^s(\R)$ for some $b > 0$.
We find that such solitary traveling waves exist for special choices of the wave speed $c$, the resonator mass $\mu$, and the bead-resonator spring constant $\kappa$, corresponding to two different ``material limits'' in the lattice, which we now describe.

\subsubsection{Small resonator limit}\label{small res limit}
We fix the traveling wave speed $c$ and the bead-resonator spring constant $\kappa$ and send $\mu \to 0$.
At $\mu = 0$ we have removed the resonator entirely and reduced the MiM lattice to a monatomic Fermi-Pasta-Ulam-Tsingou (FPUT) lattice~\cite{fput-original, dauxois} with nearest-neighbor interaction potential given by $V$.

This is naturally reflected in the traveling wave equations.
Set $\mu=0$ in~\eqref{ii} to see that $\rho_2 = 0$ is a solution, and then the first equation~\eqref{i} reads simply 
\begin{equation}\label{original monatomic fput eqn}
c^2 \rho_1'' - \delta^2 V'(\rho_1) 
= 0.
\end{equation}
This is the governing equation for traveling wave solutions of monatomic FPUT.
There are a host of different sorts of solutions to this equation, depending, in particular, on how the potential $V$ is chosen.
We are especially interested in solitary wave solutions like those developed by Friesecke and Wattis in~\cite{friesecke-wattis} and Friesecke and Pego in~\cite{friesecke-pego1}; see Section~\ref{Main Results} for further comments on the lattices in these papers.
Do these solitary waves persist for small $\mu > 0$?

\subsubsection{Stiff internal spring limit}\label{stiff sub}
We again fix the traveling wave speed but now also fix the resonator mass $\mu$ and send $\kappa \to \infty$.
A system with completely stiff internal springs corresponds to putting ``$\kappa = \infty$'' in~\eqref{traveling}. 
As in the previous special case, this should reduce the problem to a standard FPUT lattice wherein the mass of the particles is the sum of that of the outer shells and inner resonators, which is $1+\mu$.

To see this at the level of~\eqref{traveling}, put $\rho_3 = \kappa \rho_2$ so that~\eqref{traveling} become
\begin{subequations}\label{traveling-scaled}%
\be\label{iii}
c^2 \rho_1'' - \delta^2 V'(\rho_1)-  \delta \rho_3 
= 0
\ee
\be\label{iv}
c^2 \mu \kappa^{-1} \rho_3'' + (1+\mu) \rho_3 + \mu \delta V'(\rho_1)
= 0.
\ee
\end{subequations}
Then letting $\kappa = \infty$, we see that~\eqref{iv} implies that 
\[
\rho_3 
= -\frac{\mu}{1+\mu}\delta V'(\rho_1).
\] 
Substituting this into~\eqref{iii} then gets us
\[
c^2 \rho_1'' - \delta^2 V'(\rho_1)+\left(\frac{\mu}{1+\mu}\right)\delta^2 V'(\rho_1)
= 0
\]
or rather, after some easy algebra,
\[
(1+\mu) c^2 \rho_1'' - \delta^2 V'(\rho_1)
=0.
\]
This is, again, is the governing equation for traveling wave solutions of the monatomic FPUT lattice, though now the mass of the particles is, as expected, $1+\mu$. 
Do solutions of these equations persist when $\kappa$ is merely huge?

\subsection{Existing results on MiM traveling waves}
Our motivation for studying MiM lattices, and in particular the incarnation of the lattice discussed in Section~\ref{tw problem intro}, comes from a variety of recent results on MiM traveling waves, which we now summarize briefly.
We first observe that the majority of these results use a Hertzian potential to govern the nearest-neighbor interaction $V$ of the beads.
That is, for $r \in \R$ set
\[
[r]_+
:= \begin{cases}
r, &r \ge 0 \\
0, &r < 0
\end{cases}
\]
and, for some $\ep_0 \ge 0$ and $p >1$, take the potential $V$ to satisfy
\begin{equation}\label{Hertzian defn}
V'(r)
= [\ep_0+r]_+^p.
\end{equation}
The case $\ep_0 > 0$ corresponds to a ``precompressed'' lattice.

There are a number of results on solitary waves in monatomic FPUT lattices with Hertzian nearest-neighbor interactions.
The traveling wave equation for these lattices is just~\eqref{original monatomic fput eqn} with $V$ given by~\eqref{Hertzian defn}.
In the absence of precompression ($\ep_0 = 0$), English and Pego~\cite{english-pego} prove that all traveling wave solutions to this problem that vanish at infinity do so exponentially fast; the existence of such traveling wave solutions is a straightforward consequence of the solitary wave theory of Friesecke and Wattis~\cite{friesecke-wattis}.
Stefanov and Kevrekidis have considered both the Hertzian monatomic lattice without precompression~\cite{sk12} and with precompression~\cite{sk13}.
In each case they show that the solitary waves are non-decreasing on $(-\infty,0)$ and non-increasing on $(0,\infty)$; in the latter paper allowing precompression, they obtain solitary waves for wave speeds sufficiently large, depending on the constants $\ep_0$ and $p$ from~\eqref{Hertzian defn}.

The situation for mass-in-mass lattices with Hertzian nearest-neighbor interactions between the beads is somewhat different.
An ample body of numerical evidence~\cite{kev-vain-et-al, woodpile, ksx, xu, vorotnikov-et-al} indicates that the generic traveling wave solution for a Hertzian mass-in-mass lattice (i.e, the solutions to~\eqref{traveling} with $V$ satisfying~\eqref{Hertzian defn} for some $\ep_0$ and $p$) is not a solitary wave but a nanopteron.
The nanopteron traveling wave profile does not vanish at infinity like the solitary wave but instead asymptotically approaches a small-amplitude, high-frequency periodic oscillation~\cite{boyd}.

The existence of nanopterons in a different lattice model, the diatomic FPUT lattice, has recently been confirmed under a variety of different material limits.
The diatomic FPUT lattice consists of an infinite chain of particles of alternating masses connected by identical springs whose potential $V$, after normalization, satisfies
\begin{equation}\label{FPUT potential}
V'(r) = r + r^2.
\end{equation}
Faver and Wright~\cite{faver-wright} found nanopteron traveling waves in the KdV-long wave limit, and Hoffman and Wright~\cite{hoffman-wright} constructed them for the small mass limit; Faver~\cite{faver-spring-dimer} has also constructed nanopterons for the related long wave limit for FPUT lattices with alternating springs and constant masses. 
Faver and Hupkes~\cite{faver-hupkes} found a related breed of ``generalized'' or ``nonlocal'' solitary waves, called micropterons, in diatomic FPUT lattices in the equal mass limit.
All of these FPUT articles rely on an ansatz initially proposed by Beale~\cite{beale91} for a capillary-gravity water wave problem; later, Amick and Toland~\cite{amick-toland} adapted Beale's ansatz for a singularly perturbed KdV-type ODE, and recently Johnson and Wright~\cite{johnson-wright} applied it to a singularly perturbed Whitham equation.
We note that nanopteron traveling waves are not limited to FPUT lattices with only nearest neighbor interaction.
Iooss and Kirchg\"{a}ssner~\cite{ik} combine the techniques of spatial dynamics, center manifold theory, normal form analysis, and oscillatory integral analysis due to Lombardi~\cite{lombardi-tome} to find nanopterons in a monatomic lattice with an on-site potential, while Venney and Zimmer~\cite{venney-zimmer} use related methods to construct nanopterons in a monatomic lattice with nearest and next-to-nearest neighbor interactions.
Jayaprakash, Valakis and Starosvetsky~\cite{jayaprakash-et-al} study the so-called the $1:N$ dimer lattice, which is formed by alternating ``heavy'' beads with a succession of $N \ge 1$ ``light'' beads (thus the $1:1$ lattice is diatomic); here, the particle interaction is Hertzian without precompression, and they obtain, for special values of their material parameters, solitary waves, and more generally ``near-solitary'' waves with oscillatory tails.

The primary focus of Kevrekidis, Stefanov and Xu in~\cite{ksx}, however, is not the intriguing numerical evidence that they present for nanopterons in Hertzian MiM lattices but rather a construction of solitary waves for special ``antiresonance'' values of the wave speed $c$ and resonator mass $\mu$.
(They elect to normalize $\kappa=1$.)
That is, given a wave speed $c$, they find a countable number of mass values for which the Hertzian MiM lattice bears genuine solitary waves.
In the next section, we re-derive this antiresonance condition.

\subsection{The Kevrekidis-Stefanov-Xu antiresonance condition revisited}
Recall that we are interested in solutions of the traveling wave system~\eqref{traveling} which are spatially localized in the sense that they tend to zero as $|x| \to\infty$. 
We note that~\eqref{ii} is nothing more than the equation of a driven simple harmonic oscillator with ``natural frequency'' 
\begin{equation}\label{omega-defn}
\omega_{c,\mu,\kappa}
:=\sqrt{ \frac{\kappa(1+\mu)}{c^2 \mu}}.
\end{equation}

So, if we specify zero initial conditions at $x = -\infty$, variation of parameters implies
\be\label{vop}
\rho_2(x)
=-\frac{1}{c^2 \omega_{c,\mu,\kappa}} \int_{-\infty}^x \sin(\omega_{c,\mu,\kappa} (x-y)) \delta V'(\rho_1)(y) dy.
\ee
We rewrite this as
\begin{equation}\label{rho2 via VoP}
\rho_2(x) 
= -\frac{1}{c^2 \omega_{c,\mu,\kappa}} \Im \left(e^{i \omega_{c,\mu,\kappa} x} \int_{-\infty}^x e^{-i\omega_{c,\mu,\kappa} y} \delta V'(\rho_1)(y) dy\right).
\end{equation}
If we are to have $\rho_2(x) \to 0$ as $x \to \infty$, we must have
$$
\int_{-\infty}^\infty e^{-i\omega_{c,\mu,\kappa} y} \delta V'(\rho_1)(y) dy 
= 0.
$$
Thus, necessarily,
\be\label{arc1}
\hat{ \delta V'(\rho_1)}(\omega_{c,\mu,\kappa}) 
= 0.
\ee
This says the driver of the simple harmonic oscillator has no spectral content at the natural frequency; as such~\eqref{arc1} represents
an ``antiresonance condition.'' 

Meeting~\eqref{arc1} seems like a tall order since, in advance, we do not know $\rho_1$. 
But notice that $\delta$, defined in~\eqref{delta}, is a Fourier multiplier operator with symbol $2i \sin(\xi/2)$. 
Thus~\eqref{arc1} reads
\be\label{arc1 fourier}
\sin\left({\omega_{c,\mu,\kappa}}/{2}\right) \hat{V'(\rho_1)}(\omega_{c,\mu,\kappa}) 
= 0
\ee
and we see that~\eqref{arc1 fourier} is satisfied, no matter what $\rho_1$ may be, if 
\be\label{KSX}
\omega_{c,\mu,\kappa}
= \sqrt{ \kappa(1+\mu)  \over c^2 \mu}
= 2 n \pi,\quad n \in \N.
\ee
This simple algebraic relation among $c$, $\mu$ and $\kappa$ turns out to be equivalent to the antiresonance condition in~\cite{ksx} (recalling that $\kappa=1$ in that paper).

\subsection{The antiresonance-nanopteron connection}
Motivated by the successful FPUT small mass limit of Hoffman and Wright~\cite{hoffman-wright} and the numerical evidence for Hertzian MiM nanopterons discussed above, Faver studies the small resonator limit in FPUT-MiM lattices in the forthcoming paper~\cite{faver-mim}.
These are MiM lattices in which the nearest-neighbor interaction of the beads is given by~\eqref{FPUT potential}, which itself is a special case of the potentials that we consider in this paper.
In this limit, as in Section~\ref{small res limit}, the wave speed $c$ and the bead-resonator spring constant $\kappa$ are fixed, and the resonator mass $\mu$ is sent to 0.
One can view the FPUT potential~\eqref{FPUT potential} as, roughly, the Taylor expansion of the Hertzian potential with precompression, i.e., taking $\ep_0 > 0$ in~\eqref{Hertzian defn}.

Using Beale's nanopteron ansatz and methods similar to those in~\cite{hoffman-wright}, one can construct nanopteron traveling waves in FPUT-MiM lattices in the limit as $\mu \to 0$ for all but a ``small'' number of resonator masses $\mu$.
More precisely, one has solutions to~\eqref{traveling} of the form
\begin{equation}\label{nanos}
\rho_{1,\mu} = \varsigma_{1,\mu} + a_{\mu}\phi_{1,\mu} 
\mand
\rho_{2,\mu} = \varsigma_{2,\mu} + a_{\mu}\phi_{2,\mu},
\end{equation}
where $\varsigma_{j,\mu}$ is exponentially localized, $a_{\mu} \in \R$ and $\phi_{j,\mu}$ is periodic, with $\rho_1 = a_{\mu}\phi_{1,\mu}$ and $\rho_2 = a_{\mu}\phi_{2,\mu}$ also solving~\eqref{traveling}.
However, this nanopteron program breaks down at the countable number of $\mu$ satisfying~\eqref{KSX}, and so no information about the existence of nanopterons at those $\mu$ can be gleaned from it.
Our paper in part complements Faver's work by showing that at these $\mu$ there are genuine solitary waves.

An intriguing question for future consideration is, for $\mu$ small and not satisfying~\eqref{KSX}, whether or not $a_{\mu}$ from~\eqref{nanos} is zero.
If $a_{\mu} = 0$, then $\rho_{1,\mu}$ and $\rho_{2,\mu}$ are genuine MiM solitary waves.
In other words, are solitary waves limited to the antiresonance case, or can they be byproducts of the nanopteron method, too?

\section{Main Results}\label{Main Results}
We show that if the antiresonance condition~\eqref{KSX} is met, then, yes, the solitary waves obtained in the distinguished limits described in Section~\ref{dist} persist. 
To state these results precisely, we need a few definitions. We will largely work in the spaces
\[
H^s_b
:=\left\{f \in H^s : \cosh(\cdot)^bf \in H^s \right\}.
\]
These are Banach spaces with norm 
\[
\| f \|_{s,b}
:=\| \cosh(\cdot)^b f \|_{H^s}
\]
and they consist of functions which ``decay like $e^{-b|x|}$ as $|x| \to \infty$.''
We will also make use of the following subspaces of $H^s_b$: 
\[
E^s_b
:=H^s_b \cap \left\{ \text{even functions}\right\},
\quad
E^s_{b,0}
:= E^s_b \cap \left\{ f: \int_\R f(x) dx = 0 \right\}, \]
and
$$
O^s_b:=H^s_b \cap \left\{ \text{odd functions} \right\}.
$$
We also need ``a solution to perturb from'' and so we make the following definition.

\begin{definition}\label{main assumption} 
We say $c$ is a ``$V$-admissible wave speed'' if there exist $\beta>0$ and a nonzero function $\sigma \in E^2_{\beta}$ such that

\begin{enumerate}[label={(\roman*)}, ref={(\roman*)}]

\item\label{hypo 1}
$c^2 \sigma'' - \delta^2 V'(\sigma) = 0$. 

\item 
The operator
\be\label{H defn}
\H{f}
:= c^2f'' - \delta^2 V''(\sigma)f
\ee 
is a homeomorphism from $E^{2}_{\beta}$ onto $E^0_{\beta,0}$.
\end{enumerate}
Moreover if $V$ is smooth, then so is $\sigma$, and $\H$ is a homeomorphism from $E^{s+2}_\beta$ to $E^{s}_{\beta,0}$ for all $s\ge0$.
\end{definition}

Our  first result concerns the ``small internal resonator limit.''

\begin{theorem}\label{mass thm}
Assume that $V(r)$ is $C^5$, $V'(0) = 0$ and $c\ne 0$ is a $V$-admissible wave speed. 
Fix $\kappa > 0$ and suppose that
\be\label{mu-n defn}
\mu 
= \mu_{n} 
:= {\kappa \over 4 \pi^2 c^2 n^2  -\kappa}, \quad n \in \N
\ee
so that~\eqref{KSX} is satisfied.
Then there exist $C>0$ and $N \in \N$ such that $n \ge N$ implies the existence of unique functions $\rho_1$, $\rho_2 \in H_{\beta}^2$ with the following properties:

\begin{enumerate}[label={(\roman*)}, ref={(\roman*)}]

\item\label{mass concl 1}
$\rho_1$ and $\rho_2$ solve~\eqref{traveling}.

\item\label{mass concl 2}
$\rho_1$ is even and $\rho_2$ is odd.

\item\label{mass concl 3}
$\| \sigma - \rho_1\|_{2,\beta} \le C/n^2$ and  $\|\rho_2\|_{2,\beta} \le C/n$.
\end{enumerate}
Moreover, if $V$ is smooth then so are $\rho_1$ and $\rho_2$, and the estimates in~\ref{mass concl 3} improve  to 
$\| \sigma - \rho_1\|_{s,\beta}+ \|\rho_2\|_{s,\beta}\le C/n^2$ for any $s \ge 0$.

\end{theorem}

The second result concerns the ``stiff internal spring limit.''

\begin{theorem}\label{stiff thm}  
Fix $\mu>0$. 
Assume that $V(r)$ is $C^5$, $V'(0) = 0$ and $c \ne 0$ is a $(1+\mu)^{-1}V$-admissible wave speed. 
Suppose that 
\be\label{kappa-n defn}
\kappa 
= \kappa_n
:= {4\pi^2 c^2 n^2 \mu \over 1+\mu} , \quad n \in \N
\ee
so that~\eqref{KSX} is satisfied.
Then there exist $C>0$ and $N \in \N$ such that $n \ge N$ implies the existence of unique functions $\rho_1$, $\rho_2 \in H_{\beta}^2$ with the following properties:

\begin{enumerate}[label={(\roman*)}, ref={(\roman*)}]

\item\label{stiff concl 1}  
$\rho_1$ and $\rho_2$ solve~\eqref{traveling}.

\item\label{stiff concl 2}
$\rho_1$ is even and $\rho_2$ is odd.

\item\label{stiff concl 3}
$\| \sigma - \rho_1\|_{2,\beta} + \|\rho_2\|_{2,\beta} \le C/n.$
\end{enumerate}
Moreover, if $V$ is smooth then so are $\rho_1$ and $\rho_2$, and the estimate in~\ref{stiff concl 3} improves to $\| \sigma - \rho_1\|_{s,\beta}+ \|\rho_2\|_{s,\beta}\le C/n^2$ for any $s \ge 0$.
\end{theorem}

Before moving on to the proofs, we note that there are a number of choices for $c$ and $V$ where it known that $c$ is a $V$-admissible wave speed. 

\begin{enumerate}[label={(\roman*)}]

\item
Friesecke and Wattis~\cite{friesecke-wattis} prove the existence of solitary waves for potentials $V \in C^2$ that are ``superquadratic'' in the sense that $V(r)/r^2$ is strictly increasing on either $(0,\infty)$ or $(-\infty,0)$.
The wave speeds $c$ are ``supersonic'' in the sense that $c^2 > V''(0)$.
Examples of such superquadratic potentials for which their method applies include the following.

\begin{enumerate}[label=$\bullet$]

\item
The ``cubic'' FPUT potential 
\[
V(r) 
= \frac{1}{2}ar^2 + \frac{1}{6}br^3, \ a > 0, \ b \ne 0,
\]
which, along with the ``quartic'' potential below, was part of the original FPUT study~\cite{fput-original}.

\item
The ``quartic'' FPUT potential 
\[
V(r) 
= \frac{1}{2}r^2 + \frac{1}{4}br^4, \ b > 0.
\]

\item
The Toda potential $V(r) = ab^{-1}(e^{-br}-br-1)$, $ab > 0$ (see also~\cite{toda, teschl}).

\item 
The Lennard-Jones potential $V(r) = a((r+d)^{-6} + r^{-6})^2$, $a$, $d > 0$.
\end{enumerate}

\item
Friesecke and Pego~\cite{friesecke-pego1} find solitary waves in the ``near-sonic/KdV limit'' for potentials $V \in C^4$ satisfying $V(0) = V'(0) = 0$, $V''(0) > 0$ and $V'''(0) \ne 0$.
Their solitary waves have speeds $c$ satisfying $c^2$ larger than, but close to, $V''(0)$.
All of the potentials mentioned above from the Friesecke-Wattis results, except the quartic potential, fall into this category.
Hoffman and Wright~\cite{hoffman-wright} and Faver and Hupkes~\cite{faver-hupkes} prove the invertibility of $\H$ from~\eqref{H defn} under slightly different sets of assumptions for the cubic FPUT potential.
We remark that Iooss~\cite{iooss} adapts the spatial dynamics and center manifold theory of~\cite{ik} to determine {\it{all}} small bounded traveling wave solutions to the monatomic FPUT lattice for potentials $V \in C^4$, including but not limited to solitary waves.

\item
Herrmann and Matthies~\cite{herrmann-matthies-asymptotic, herrmann-matthies-uniqueness, herrmann-matthies-stability} study the model potential
\[
V(r)
= \frac{1}{m(m+1)}\left(\frac{1}{(1-r)^m} - mr-1\right), \ m > 1,
\]
in the ``high-energy limit,'' for arbitrarily large wave speeds.
In particular, the invertibility of $\H$ for this potential is discussed in~\cite{herrmann-matthies-uniqueness}.
See also the additional convex potentials described in Section 4.4 of~\cite{herrmann10}.
\end{enumerate}

\section{The Proofs of Theorems~\ref{mass thm} and~\ref{stiff thm}}

\subsection{Reduction to a single equation}
The first step in the proofs  is to use the antiresonance condition~\eqref{KSX} to develop a useful formula for $\rho_2$ in terms of $\rho_1$, thereby reducing the problem 
to a single equation for $\rho_1$.
We saw in~\eqref{rho2 via VoP} one such formula, but it will be more convenient to work on the Fourier side and express the formula in terms of Fourier multipliers.

We assume $c$, $\mu$ and $\kappa$ satisfy~\eqref{KSX} 
and apply the Fourier transform to the second traveling wave equation~\eqref{ii} to get 
\be\label{iihat}
\left(-c^2 \mu \xi^2 + \kappa(1+\mu)\right)\hat{\rho}_2(\xi) + 2 i \mu \sin(\xi/2)\hat{V'(\rho_1)}(\xi) = 0.
\ee
Isolating ${\rho}_2$ gives
\be\label{FMO reduction}
{\rho}_2 
={\Sigma}_{c,\mu,\kappa} V'(\rho_1)
\ee
where the operator $\Sigma_{c,\mu,\kappa}$ is the Fourier multiplier operator with symbol
\be\label{symbol of Sigma}
\tilde{\Sigma}_{c,\mu,\kappa}(\xi)
:={2i \mu \sin(\xi/2) \over c^2 \mu \xi^2 - \kappa (1+\mu)}.
\ee
Condition~\eqref{KSX} implies that the singularities in the multiplier due to the zeroes of the denominator  are removable; in fact, demanding that these
singularities are removable is another path to determining~\eqref{KSX}.

\begin{remark} 
Standard results about the Fourier transforms of $\sinc$ and convolutions can be used to show that the formula~\eqref{FMO reduction} 
coincides
with~\eqref{rho2 via VoP}.
\end{remark}

We put the expression for $\rho_2$ from~\eqref{FMO reduction} into the first traveling wave equation~\eqref{i}, to get
\be\label{v}
c^2 \rho_1'' - \left(\delta^2   + \kappa \delta \Sigma_{c,\mu,\kappa}\right) V'(\rho_1) 
= 0.
\ee
Kevrekidis, Stefanov and Xu also reduce their solitary wave problem in~\cite{ksx} to a single equation for $\rho_1$ by using Fourier methods.
But afterward their methods and ours diverge: they use a variational approach to solve for $\rho_1$, whereas we will construct it with a perturbation argument.
(We note that they set up a similar Fourier argument in~\cite{xks} for an MiM lattice whose bead interaction is governed by a Hertzian potential {\it{with}} precompression, unlike that in~\cite{ksx}.)

Note that $\sin(\xi/2) \tilde{\Sigma}_{c,\mu,\kappa}$ is quadratic in $\xi$ for $\xi \sim 0$, as is the multiplier for $\delta^2$. This motivates  ``regularizing''~\eqref{v} by applying
$\partial_x^{-2}$ (intepreted as a Fourier multiplier with symbol $-1/\xi^2$) to it to get:
\be\label{vi}
c^2 \rho_1- \left( \partial_x^{-2} \delta^2   + \Gamma_{c,\mu,\kappa}\right) V'(\rho_1)
= 0.
\ee
Here $\Gamma_{c,\mu,\kappa} := \kappa \partial_x^{-2} \delta \Sigma_{c,\mu,\kappa}$ is a Fourier multiplier with symbol
\be\label{symbol of Gamma}
\tilde{\Gamma}_{c,\mu,\kappa}(\xi)={4 \kappa \mu \sin^2(\xi/2) \over  \xi^2 (c^2 \mu\xi^2 - \kappa (1+\mu))}= {4 \kappa \over c^2} {\sin^2(\xi/2) \over \xi^2( \xi^2 - \omega^2_{c,\mu,\kappa})}.
\ee
Likewise the multiplier for $\partial_x^{-2} \delta^2$ is $\sinc^2(\xi/2)$.
Equation~\eqref{vi} is the main version of the system we will be working to solve.

\subsection{Estimates on $\Sigma_{c,\mu,\kappa}$ and $\Gamma_{c,\mu,\kappa}$.}
A routine partial fractions expansion shows that
$$
{ 1 \over \xi^2 (\xi^2 - \omega^2)}= -{1 \over \omega^2 \xi^2} + {1 \over 2 \omega^3(\xi-\omega)}-{1 \over 2 \omega^3(\xi+\omega)}.
$$
Thus we have
$$
{\tilde{\Gamma}_{c,\mu,\kappa}(\xi)} 
=  {4 \kappa \over c^2}
\left(
-{ \sin^2\left({\xi /2}\right) \over \omega_{c,\mu,\kappa}^2 \xi^2}
+{ \sin^2\left({\xi /2}\right) \over 2 \omega^3_{c,\mu,\kappa}(\xi- \omega_{c,\mu,\kappa})} 
-{ \sin^2\left({\xi /2}\right) \over 2 \omega^3_{c,\mu,\kappa}(\xi+ \omega_{c,\mu,\kappa})} 
\right) .
$$

If we assume the antiresonance condition~\eqref{KSX}, then the addition of angles formula tells us that 
$
\sin\left(\xi/2\right) 
= \sin\left( (\xi-\omega_{c,\mu,\kappa})/2\right) \cos\left( \omega_{c,\mu,\kappa}/2\right).
$
Using this in the last expression gives, after some algebra:
$$
{\tilde{\Gamma}_{c,\mu,\kappa}(\xi)} 
=  -{ \kappa \over  c^2 \omega^2_{c,\mu,\kappa}} { \sinc^2\left({\xi / 2}\right)}+ \tilde{\Psi}_{c,\mu,\kappa}(\xi)
$$
where
$$
\tilde{\Psi}_{c,\mu,\kappa}(\xi):= { \kappa \over c^2 \omega^3_{c,\mu,\kappa}}
{\sin\left({\xi / 2}\right) \cos\left({\omega_{c,\mu,\kappa} / 2}\right)} \left[  \sinc\left({(\xi- \omega_{c,\mu,\kappa}) / 2}\right) 
-\sinc\left({(\xi+ \omega_{c,\mu,\kappa}) / 2}\right) \right].
$$
Therefore~\eqref{KSX} implies
\be\label{KSX implies}
\Gamma_{c,\mu,\kappa} = -{\kappa \over c^2 \omega^2_{c,\mu,\kappa}} \partial_x^{-2} \delta^2 + \Psi_{c,\mu,\kappa}
\ee
where $\Psi_{c,\mu,\kappa}$ is a Fourier multiplier with symbol $\tilde{\Psi}_{c,\mu,\kappa}(\xi)$.

A very similar line of reasoning shows that $\eqref{KSX}$ implies
\be\label{KSX implies 2}
\tilde{\Sigma}_{c,\mu,\kappa}(\xi) 
= {i  \over 2 c^2 \omega_{c,\mu,\kappa}}\cos(\omega_{c,\mu,\kappa}/2) \left[  \sinc\left({(\xi- \omega_{c,\mu,\kappa}) / 2}\right) 
-\sinc\left({(\xi+ \omega_{c,\mu,\kappa}) / 2}\right) \right].
\ee

With~\eqref{KSX implies} and~\eqref{KSX implies 2} in hand, we can call on the following result (see~\cite{beale80}, for instance) for quantitative estimates:

\begin{lemma}\label{FMO estimate} 
Let $M$ be a Fourier multiplier operator and assume that its symbol $\tilde{M}(z)$ is analytic in the set $\left\{ |\Im(z)| \le b \right\}$. 
Then
$$
\| M \|_{H^s_b \to H^{s-r}_b} \le C \sup_{\xi \in \R} \left\vert (1+|\xi|)^{-r} \tilde{M}(\xi \pm  i b)\right \vert.
$$
The constant $C >0$ depends only on $r$, not on $M$ or $b$ or $s$.
\end{lemma}

The $\sinc$ and $\sin$ functions are entire and are also bounded in horizontal strips of the complex plane, {\it i.e.} $\sup_{|\Im(z)| \le b} \left(  |\sin(z)| + |\sinc(z)|\right) \le C_b$ for some  $C_b>0$.  
Moreover, the mulitpliers for $\Gamma_{c,\mu,\kappa}$ and $\Psi_{c,\mu,\kappa}$ are even in $\xi$ and as such map even functions to even functions. 
The multiplier for  $\Sigma_{c,\mu,\kappa}$ is odd and thus flips parity. 
Therefore we can conclude from~\eqref{KSX implies} and~\eqref{KSX implies 2}:

\begin{lemma}\label{main lemma}
For all $b > 0$  there exists a constant $C_b > 0$ such that if  $c$, $\mu$ and $\kappa$ are positive and meet~\eqref{KSX}, then 
$$
\| \Sigma_{c,\mu,\kappa} \|_{E^s_b \to O^s_b}  
\le {C_b \over  c^{2}  \omega_{c,\mu,\kappa}},
\quad 
\| \Gamma_{c,\mu,\kappa} \|_{E^s_b \to E^s_b}  
\le {C_b \kappa \over  c^{2}  \omega^2_{c,\mu,\kappa}} 
\mand 
\| \Psi_{c,\mu,\kappa} \|_{E^s_b \to E^s_b}  
\le {C_b \kappa \over  c^{2}  \omega^3_{c,\mu,\kappa}} .
$$
\end{lemma}

We need a few more technical estimates. 
For $\omega  \in \R$ let $S_\omega$ be the Fourier multiplier with symbol $\tilde{S}_\omega(\xi):=\sinc((\xi-\omega)/2)$.
Note that these are constituent elements of $\partial_x^{-2} \delta^2$, $\Sigma_{c,\mu,\kappa}$ and $\Psi_{c,\mu,\kappa}$. 
We have:

\begin{lemma}\label{S est} For all $b \ge 0$ there exists $C_b>0$  such that, for all $s \ge 0$,
$$
\| S_\omega\|_{H^{s}_b \to H^{s-1}_b} \le C_b/(1+|\omega|) \mand \| S_\omega\|_{H^{s}_b \to H^{s+1}_b} \le C_b(1+|\omega|).
$$
\end{lemma}

\begin{proof}
The multiplier for $S_\omega$ is entire and from Lemma~\ref{FMO estimate} we know that
$$
\| S_\omega\|_{H^{s}_b \to H^{s\pm 1}_b} \le C \sup_{\xi \in \R} \left\vert (1+|\xi|)^{\pm 1} \sinc((\xi+ib-\omega)/2) \right \vert.
$$ 
Elementary considerations show that there is a constant $C_b >0$ for which 
\[
|\sinc((\xi+ib-\omega)/2)|
\le C_b(1+(\xi -\omega)^2)^{-1/2}
\]
holds for all $\xi \in \R$. 
Similarly $(1+|\xi|)^{-1} \le C(1+\xi^2)^{-1/2}$ for all $\xi \in \R$. 
Then methods from differential calculus show that $\max_{\xi \in \R} (1+(\xi -\omega)^2)^{-1/2}(1+\xi^2)^{-1/2} \le  C/(1+|\omega|)$.  
This gives the first estimate in the lemma.

As for the other estimate, we know that $(1+|\xi|) \le C(1+\xi^2)^{1/2}$ for all $\xi \in \R$. 
And differential calculus again tells us that 
$\max_{\xi \in \R} (1+(\xi -\omega)^2)^{-1/2}(1+\xi^2)^{1/2} \le C(1+ |\omega|)$.  
This completes the proof.
\end{proof}

Lemma~\ref{S est} immediately implies:

\begin{lemma}\label{refined lemma}
For all $b > 0$  there exists a constant $C_b > 0$ such that if $c$, $\mu$ and $\kappa$ are positive and meet~\eqref{KSX}, then 
\be\label{small estimates}
\| \Sigma_{c,\mu,\kappa} \|_{E^s_b \to O^{s-1}_b}  
\le {C_b \over  c^{2}  \omega_{c,\mu,\kappa}^2} \mand 
\| \Psi_{c,\mu,\kappa} \|_{E^s_b \to E^{s-1}_b}  
\le {C_b \kappa \over  c^{2}  \omega^4_{c,\mu,\kappa}} .
\ee
Likewise
\be\label{smooth estimates}
\| \Sigma_{c,\mu,\kappa} \|_{E^s_b \to O^{s+1}_b}  
\le {C_b \over  c^{2} } \mand 
\| \Psi_{c,\mu,\kappa} \|_{E^s_b \to E^{s+1}_b}  
\le {C_b \kappa \over  c^{2}  \omega^2_{c,\mu,\kappa}} .
\ee
\end{lemma}

Lemma~\ref{refined lemma} imputes a sort of ``Jeckyll \& Hyde'' identity to the operators. 
The estimates in~\eqref{small estimates} tell us that these operators are small when $\omega_{c,\mu,\kappa}$ is large, but that this smallness comes at a cost in regularity.
The estimates in~\eqref{smooth estimates} tell us that these operators are smoothing, but in this case they are not (as) small. 

\subsection{A perturbation lemma} 
Note that~\eqref{vi} is a perturbation of 
$c^2 \rho_1- \partial_x^{-2} \delta^2 V'(\rho_1) =0$.
It is an easy exercise to show the following.

\begin{lemma} If $c$ is a $V$-admissible wave speed, then:

\begin{enumerate}[label={(\roman*)}]

\item 
$c^2 \sigma - \partial_x^{-2} \delta^2 V'(\sigma) = 0$ and 

\item 
the operator
\be\label{L-c defn}
\L{f}
:= c^2f - \partial_x^{-2} \delta^2 V''(\sigma)f
\ee  
is a homeomorphism from $E^{2}_{\beta}$ to itself. 
\end{enumerate}
Moreover if $V$ is smooth then so is $\sigma$ and $\L$ is a homeomorphism from $E^{s}_{\beta}$ to itself for all $s \ge 0$.
\end{lemma}

Since we now know from Lemma~\ref{main lemma} how big $ \Gamma_{c,\mu,\kappa}$ is in~\eqref{vi}, we will use the following result to find solutions.

\begin{lemma} \label{routine}
Assume that $V(r)$ is $C^5$ and that  $c$ is a $V$-admissible wave speed. 
Then there are positive constants $\alpha$ and $\tau$ so that for each linear operator $B:E^{2}_{\beta} \to E^2_{\beta}$ with  $\| B \|_{E_{\beta}^2 \to E^2_{\beta}} \le \tau$, there is a unique function $\rho \in  E^2_{\beta}$ for which
$$
c^2 \rho - (\partial_x^{-2} \delta^2 +B)V'(\rho) = 0 
\mand \| \sigma-\rho \|_{2,\beta} \le \alpha\|  B \|_{E_{\beta}^2 \to E^2_{\beta}}.
$$
\end{lemma}

\begin{proof} 
Let $\rho = \sigma + \eta$ and substitute this into $c^2 \rho - (\partial_x^{2-}\delta^2 +B)V'(\rho) = 0.$ 
Recalling the definition of $\L$ from~\eqref{L-c defn}, one finds that $\eta$ solves
$$
\L \eta = \partial_x^{-2} \delta^2 \left( V'(\sigma + \eta) -V'(\sigma)-  V''(\sigma) \eta\right)  +B V'(\sigma + \eta).
$$
Since $V'(0) = 0$, we have $V'(f) \in E_{\beta}^2$ for any $f \in E_{\beta}^2$ (see Lemma~\ref{subs lemma} below).
The operators $\partial_x^{-2} \delta^2$ and $B$ both map $E^2_{\beta}$ into  itself and so we can invert $\L$ to get
$$
\eta = \L^{-1} \left( \partial_x^{-2}\delta^2 \left( V'(\sigma + \eta) -V'(\sigma)-  V''(\sigma) \eta\right)\right)  +\L^{-1}B V'(\sigma + \eta) =:R(\eta).
$$

We now show that $R$ is a contraction on an appropriate closed ball in $E_{\beta}^2$.  
We first use the the boundedness of $\L^{-1}$, $\partial_x^{-2} \delta^2$ and $B$ to get
\be\label{R map 0}
\| R (\eta)\|_{2,\beta} \le C\|V'(\sigma+\eta) - V'(\sigma) - V''(\sigma) \eta\|_{2,\beta} + C\| B \|_{E_{\beta}^2 \to E^2_{\beta}} \|V'(\sigma+\eta)\|_{2,\beta}
\ee
and
\be\label{R lip 0}\begin{split}
&\| R (\eta) - R(\gamma)\|_{2,\beta}\\ \le &C\|\left( V'(\sigma+\eta) - V'(\sigma) - V''(\sigma) \eta\right) -\left( V'(\sigma+\gamma) - V'(\sigma) - V''(\sigma) \gamma\right)\|_{2,\beta}\\  &+C\| B \|_{E_{\beta}^2 \to E^2_{\beta}} \|V'(\sigma+\eta)-V'(\sigma+\gamma)\|_{2,\beta}.
\end{split}
\ee

Next we need the following substitution operator estimates, which can be deduced from Propositions B.3 and B.6 in~\cite{faver-spring-dimer}.

\begin{lemma} \label{subs lemma} 
Fix $b\ge 0$, $P>0$ and $s \ge 1$.

\begin{enumerate}[label={(\roman*)}]

\item
Suppose that $F \colon \R \to \R$ is $C^{s+1}$ and $F(0) = 0$.
Then $F(f) \in E_b^s$ and there is $C > 0$ such that if $\| f \|_{s,b} \le P$, then $\| F(f) \|_{s,b} \le C\| f \|_{s,b}$.

\item
Suppose that $F: \R \to \R$ is $C^{s+2}$.  
For $f,g \in E^s_b$ let $L_f(g):=F(f+g) - F(f)$ and $Q_f(g):=F(f+g) - F(f) - F'(f)g$.
Then there exists $C>0$ such that 
for all $f,g_1,g_2 \in H^s_b$ with $\| f\|_{s,b}, \| g_1\|_{s,b}, \|g_2\|_{s,b} \le P$ we have
$$
\| L_f(g_1)-L_f(g_2))\|_{s,b} \le C\| g_1-g_2 \|_{s,b}
$$
and
$$
\| Q_f(g)  -Q_f(g_2) \|_{s,b} \le C\left(\|g_1\|_{s,b} + \|g_2\|_{s,b} \right)\| g_1-g_2\|_{s,b}.
$$
\end{enumerate}
\end{lemma}

Since we have assumed that $V$ is $C^5$ and $V'(0)=0$ we can apply these estimates in Lemma~\ref{subs lemma} to the right hand sides of~\eqref{R map 0} and~\eqref{R lip 0} to get
\be\label{R map 1}
\|R(\eta)\|_{2,\beta}
\le {1 \over 2} \alpha\left( \|\eta\|^2_{2,\beta}+ \|  B \|_{E^2_{\beta} \to E^2_{\beta}}\right) 
\ee
and
\be\label{R lip 1}
\|R(\eta) - R(\gamma)\|_{2,\beta}
\le {1 \over 2} \alpha\left( \| \eta\|_{2,\beta}+\| \gamma\|_{2,\beta}+ \|  B \|_{E^2_{\beta} \to E^2_{\beta}} \right)\| \eta - \gamma\|_{2,\beta},
\ee
for some $\alpha > 0$ 
so long as  $\eta,\gamma \in E^2_\beta$ with $\| \eta \|_{2,\beta},\| \gamma \|_{2,\beta} \le 1$. 

Next, define 
\[
\tau
:= \min\left\{\frac{1}{\alpha}, \frac{1}{\alpha^2}, \frac{1}{2 \alpha^2+\alpha}\right\}.
\]
Assume $\|  B \|_{E^2_{\beta} \to E^2_{\beta}} \le \tau$ and $\eta,\gamma \in E^2_\beta$ with $\|\eta\|_{2,\beta},\|\gamma\|_{2,\beta}\le \alpha \|  B \|_{E^2_{\beta} \to E^2_{\beta}}$. 
Thus~\eqref{R map 1} implies
\[
\|R(\eta)\|_{2,\beta}
\le {1 \over 2} \alpha\left( \alpha^2 \|  B \|_{E^2_{\beta} \to E^2_{\beta}} + 1\right)\|  B \|_{E^2_{\beta} \to E^2_{\beta}}
\le \alpha\| B \|_{E_{\beta}^2 \to E^2_\beta}
\]
and~\eqref{R lip 1} implies
\[
\|R(\eta) - R(\gamma)\|_{2,\beta}
\le  {1 \over 2} \left(2 \alpha^2 + \alpha \right)\|  B \|_{E^2_{\beta} \to E^2_{\beta}} \| \eta - \gamma\|_{2,\beta}
\le \frac{1}{2}\| \eta - \gamma\|_{2,\beta}.
\]
That is, $R$ is a contraction on the ball of radius $\alpha \|  B \|_{E^2_{\beta} \to E^2_{\beta}}$ centered at the origin in $E^2_{\beta}$, and so Banach's fixed point theorem provides us with the solution of our equation.
\end{proof}

\subsection{The final steps for the proof of Theorem~\ref{mass thm}}
Suppose that $c\ne0$ is $V$-admissible and take $\sigma$ and $\beta$ as in Definition~\ref{main assumption}.
Fix $\kappa>0$. Take $\mu = \mu_n$ as in the statement of Theorem~\ref{mass thm} so that~\eqref{KSX} is satisfied.
Thus we have $\omega_{c,\mu_n,\kappa} = 2 n \pi$ for $n \in \N$. Let $B_n:=\Gamma_{c,\mu_n,\kappa}$. 
Lemma~\ref{main lemma} then implies that
$$
\| B_n \|_{E^s_\beta \to E^s_\beta} \le  {C_\beta \kappa/c^2 \omega_{c,\mu_n,\kappa}^2} \le {C/n^2}.
$$
And so for $n$ sufficiently large, $\|B_n\|_{E^s_\beta \to E^s_\beta}$ falls below the threshold $\tau$ from Lemma~\ref{routine} and that result conjures up $\rho_1 \in E^2_\beta$ for which $c^2 \rho_1 - (\partial_x^{-2} \delta^2 + B_n) V'(\rho_1) = 0$ and
\be\label{eta est}
\|\sigma - \rho_1\|_{2,\beta} \le \alpha \| B_n \|_{E^s_\beta \to E^s_\beta} \le C/n^2.
\ee

Next put $\rho_2 = \Sigma_{c,\mu_n,\kappa} V'(\rho_1)$ as in~\eqref{FMO reduction}.
From Lemma~\ref{main lemma} we know that 
\[
\|  \Sigma_{c,\mu_n,\kappa} \|_{E^2_\beta \to O^2_\beta} 
\le C/\omega_{c,\mu_n,\kappa} 
\le C/n.
\]
Since $\rho_1$ is even, so is $V'(\rho_1)$ and  we have $\rho_2 \in O^2_\beta$ with $\| \rho_2 \|_{2,\beta} \le C \| V'(\rho_1)\|_{2,\beta}/n$. 
Then we use the substitution operator estimates in Lemma~\ref{subs lemma} to get $ \| V'(\rho_1)\|_{2,\beta} \le C\| \rho_1 \|_{2,\beta} \le C$. 
Therefore $\| \rho_2 \|_{2,\beta} \le C /n$.
Unraveling the steps which lead from~\eqref{traveling} to~\eqref{FMO reduction} and~\eqref{vi} shows that $\rho_1$ and $\rho_2$ solve~\eqref{traveling} and we have conclusions~\ref{mass concl 1},~\ref{mass concl 2} and~\ref{mass concl 3} of Theorem~\ref{mass thm}.

Now we prove the amplification of the result when $V$ is smooth; the main additional piece of information here is that $\sigma$ is smooth. 
The first step is to prove the stronger estimate for $\|\rho_2\|_{2,\beta}$.
Since $\rho_2$ satisfies~\eqref{FMO reduction} we have
\be\label{rho2 eqn}
\rho_2 = \Sigma_{c,\mu_n,\kappa} V'(\sigma) + \Sigma_{c,\mu_n,\kappa} \left[ V'(\rho_1) - V'(\sigma) \right].
\ee
Taking the $H^2_\beta$ norm of both sides gives:
$$
\| \rho_2 \|_{2,\beta} \le \|  \Sigma_{c,\mu_n,\kappa} V'(\sigma)\|_{2,\beta} + \|\Sigma_{c,\mu_n,\kappa} \left[ V'(\rho_1) - V'(\sigma) \right]\|_{2,\beta} .
$$
On the first term we use the estimate~\eqref{small estimates} for $\Sigma_{c,\mu,\kappa}$ in Lemma~\ref{refined lemma} to get
$$ 
\|  \Sigma_{c,\mu_n,\kappa} V'(\sigma)\|_{2,\beta} \le C \omega_{c,\mu_n,\kappa}^{-2} \| V'(\sigma)\|_{3,\beta} \le C/n^{2}.
$$
Note that it was essential here that $\sigma$ be smooth. 
As for the second term, the estimate for $\Sigma_{c,\mu,\kappa}$ in Lemma~\ref{main lemma} gives
$$
\|\Sigma_{c,\mu_n,\kappa} \left[ V'(\rho_1) - V'(\sigma) \right]\|_{2,\beta}\le C \omega_{c,\mu_n,\kappa}^{-1} \|  V'(\rho_1) - V'(\sigma) \|_{2,\beta} \le 
C\|  V'(\rho_1) - V'(\sigma) \|_{2,\beta}/n.
$$
Substitution operator estimates from Lemma~\ref{subs lemma}, together with~\eqref{eta est}, give us 
\[
\|  V'(\rho_1) - V'(\sigma) \|_{2,\beta} 
\le C \| \rho_1 - \sigma\|_{2,\beta} 
\le C/n^2.
\]
Thus we have $\| \rho_2 \|_{2,\beta} \le C/n^2$.

The completion of the proof is done via a bootstrap argument.  
Let $\eta:=\rho_1 - \sigma$. 
Since $c^2 \sigma'' = \delta^2 V'(\sigma)$ and $\rho_1$ solves~\eqref{i}, we see that 
$$
\eta'' = c^{-2} \delta^2 [V'(\rho_1) - V'(\sigma)] + c^{-2}\kappa \delta \rho_2.
$$
This implies, using the same sort of reasoning as above, that
\be\label{ind 1}
\|\eta\|_{s+2,\beta} \le C \| \eta \|_{s,\beta} + C\|\rho_2 \|_{s,\beta}.
\ee

Last, we estimate the $\rho_2$ term in~\eqref{ind 1}.
If we apply the $\Sigma_{c,\mu,\kappa}$ estimate~\eqref{small estimates} to the first term in~\eqref{rho2 eqn}, again relying on the smoothness of $\sigma$, and crudely estimate the second using the $\Sigma_{c,\mu,\kappa}$ estimate~\eqref{smooth estimates} and the substitution operator estimates, we get
\be\label{ind 2}
\|\rho_2\|_{s+1,\beta} \le C/n^2 + C \| \eta\|_{s,\beta}.
\ee
Since we know that $\|\eta\|_{2,\beta} = \|\rho_1 - \sigma\|_{2,\beta} \le C/n^2$ and $\|\rho_2 \|_{2,\beta} \le C/n^2$, the estimates~\eqref{ind 1} and~\eqref{ind 2} imply by induction that $\|\rho_1 - \sigma\|_{s,\beta} + \|\rho_2 \|_{s,\beta} \le C/n^2$ for all $s\ge0$. 
We have completed the proof of Theorem~\ref{mass thm}.

\subsection{The final steps for the proof of Theorem~\ref{stiff thm}}
Fix $\mu>0$. 
Suppose that $c\ne0$ is $(1+\mu)^{-1}V$-admissible and take $\sigma$ and $\beta$ as in that definition.
Take $\kappa = \kappa_n$ as in the statement of Theorem~\ref{stiff thm} so that~\eqref{KSX} is satisfied.
Thus we have $\omega_{c,\mu,\kappa_n} = 2 n \pi$ for $n \in \N$. 

In this case we have $\kappa_n/c^2 \omega_{c,\mu,\kappa_n}^2 = \mu/(1+\mu)$ and so~\eqref{KSX implies} tells us
$$
\Gamma_{c,\mu,\kappa_n} = -{\mu/(1+\mu)} \partial_x^{-2} \delta^2 + \Psi_{c,\mu,\kappa_n}.
$$
With this,~\eqref{vi} becomes, after simplification,
\be\label{vi stiff}
c^2 \rho_1 - (\partial_x^{-2} \delta^2  + (1+\mu) \Psi_{c,\mu,\kappa_n} ) \left( {1 \over 1+\mu} V'(\rho_1) \right) =0.
\ee
Now put $B_n:=(1+\mu) \Psi_{c,\mu,\kappa_n}$. 
Lemma~\ref{main lemma} implies
\[
\| B_n \|_{E^2_\beta \to E^2_\beta} 
\le C_\beta \kappa_n(1+\mu) /c^2 \omega_{c,\mu,\kappa_n}^3 
\le C/n.
\]
Thus for $n$ big enough we can call on Theorem~\ref{routine} and get $\rho_1 \in E^2_\beta$ for which~\eqref{vi stiff} and
\be\label{eta est stiff}
\|\sigma - \rho_1\|_{2,\beta} \le \alpha \| B_n \|_{E^s_\beta \to E^s_\beta} \le C/n.
\ee
We put $\rho_2 = \Sigma_{c,\mu,\kappa_n} V'(\rho_1)$ as in~\eqref{FMO reduction} and following the same steps as in the previous section we find $\| \rho_2 \|_{2,\beta} \le C /n$. 
Likewise we get conclusions~\ref{stiff concl 1},~\ref{stiff concl 2} and~\ref{stiff concl 3} of Theorem~\ref{stiff thm}.

The improvements when $V$ is smooth are a bit trickier in this setting. 
As above, if $V$ is smooth, so is $\sigma$ and this will drive the argument. 
Letting $\eta := \rho_1 - \sigma$ we find that $\eta$ solves (as in the proof of Lemma~\ref{routine})
$$
\eta = (1+\mu)^{-1} \L^{-1} \left( \partial_x^{-2}\delta^2 \left( V'(\sigma + \eta) -V'(\sigma)-  V''(\sigma) \eta\right)\right)  + \L^{-1} \Psi_{c,\mu,\kappa_n} V'(\sigma + \eta).
$$
Take the $H^s_\beta$ (with $s\ge 3$) norm of both sides to get
\be\label{1 and 2}
\| \eta \|_{s,\beta} \le C\bunderbrace{ \| \partial_x^{-2} \delta^2  \left( V'(\sigma + \eta) -V'(\sigma)-  V''(\sigma) \eta\right)\|_{s,\beta}}{I} + C\bunderbrace{\|\Psi_{c,\mu,\kappa_n} V'(\sigma + \eta)\|_{s,\beta}}{\II} .
\ee
We have used the fact that $\L$ is a homeomorphism here. 
Next, we estimate $\partial_x^{-2} \delta^2=S_0^2$ with Lemma~\ref{S est} and then use Lemma~\ref{subs lemma} to bound $I \le C \|\eta\|_{H^{s-2}_\beta}^2.$

As for $\II$, we use the triangle inequality to get
\be\label{II}
\II \le  \| \Psi_{c,\mu,\kappa_n} (V'(\sigma+\eta) - V'(\sigma)) \|_{s,{\beta}} +  \| \Psi_{c,\mu,\kappa_n} V'(\sigma) \|_{s,{\beta}}.
\ee
Using the smoothing estimate~\eqref{smooth estimates} for $\Psi_{c,\mu,\kappa}$ on the first term and the estimate~\eqref{small estimates} for $\Psi_{c,\mu,\kappa}$ on the other term gets us
\be\label{II bound 2}
\II \le C \kappa_n \omega_{c,\mu,\kappa_n}^{-2} \|V'(\sigma+\eta) - V'(\sigma) \|_{{s-1},{\beta}} + C \kappa_n \omega_{c,\mu,\kappa_n}^{-4} \|V'(\sigma) \|_{{s+1},{\beta}}
\ee
Then Lemma~\ref{subs lemma} gives
$$
\II \le C \kappa_n \omega_{c,\mu,\kappa_n}^{-2} \|\eta \|_{{s-1},{\beta}}  + C \kappa_n \omega_{c,\mu,\kappa_n}^{-4} \|\sigma \|_{{s+1},{\beta}}.
$$
Since $\kappa_n = \O(n^2)$ and $\omega_{c,\mu,\kappa_n} = \O(n)$ we have
$$
\II \le C  \|\eta \|_{{s-1},{\beta}} + C/ n^{2}. 
$$
All together then we have
$$
\| \eta \|_{s,\beta}  \le C \|\eta\|_{{s-2},\beta}^2 + C\|\eta\|_{{s-1},\beta} + C/n^2.
$$
Since we know from above that $\|\eta\|_{2,\beta} \le C/n$, a simple induction argument using the above estimate tells us that
\be\label{crude}
\| \eta \|_{s,\beta} \le C/n 
\ee
for all $s\ge2$.

Now we estimate $\II$ again, but slightly differently. Using~\eqref{small estimates} for $\Psi_{c,\mu,\kappa}$ on both terms in~\eqref{II bound 2}:
$$
\II 
\le C \kappa_n \omega_{c,\mu,\kappa_n}^{-4} \|\eta \|_{{s+1},{\beta}} + C \kappa_n \omega_{c,\mu,\kappa_n}^{-4} \|\sigma \|_{{s+1},{\beta}} 
\le C( \|\eta \|_{{s+1},{\beta}}  +1 )/n^2.
$$
Thus we have from~\eqref{1 and 2}
$$
\| \eta \|_{s,\beta}  
\le C \|\eta\|_{{s-2},\beta}^2 +C( \|\eta \|_{{s+1},{\beta}}  +1 )/n^2.
$$
Using~\eqref{crude} converts this to
$$
\| \eta \|_{s,\beta}  
\le C/n^2
$$
for all $s \ge 2$.
Since $\rho_2 = \Sigma_{c,\mu,\kappa_n} V'(\rho_1)$, it follows from this last estimate (as in the proof of Theorem~\ref{mass thm}) that $\| \rho_2\|_{s,\beta} \le C/n^2$, for all $s \ge 2$.

\section{Numerical Explorations}\label{sec:numerics}
In this section we both confirm and extend the conclusions of the paper's two main theorems. 
In particular, both theorems concern the existence of antiresonant traveling waves on particular sets of points in the parameter space, so it is of interest to determine numerically whether the conclusions remain true more generally. 

To solve for the traveling waves numerically, we fix an admissible velocity $c>1$ and an integer $n\ge 1$ and choose the parameters $\kappa$ and $\mu$ to satisfy assumption~\eqref{KSX}.  We then solve the regularized form of the equations~\eqref{vi} for $\rho_1$.
Afterward, $\rho_2$ is found by numerically evaluating~\eqref{FMO reduction}. 
As a further check on the procedure, the numerical solution is then substituted directly into system~\eqref{traveling}, and the pointwise error is found to be within the numerical tolerance of the solver. 
Subsequently, the solution is downsampled to the discrete variables $R_j$ and $r_j$ using equation~\eqref{ansatz}, and the initial value problem for the displacement coordinates is solved numerically using the sixth-order in time symplectic method of Yoshida~\cite{yoshida}.
The solutions are found to propagate without distortion over large distances. 

The result of one such numerical solution and subsequent time-integration is shown in Figure~\ref{fig:Rjt}. 
The traveling wave was computed on $-16<x\le16$ with 256 points, with parameters $c=2$, $\mu=0.4$ and $n=1$, and periodic boundary conditions. 
The downsampled values of $R_j$ and $r_j$ are used as initial conditions for the time-dependent problem in displacement coordinates. 
The solution is shown propagating one time around the periodic domain, but in simulations, we have observed the solution traveling around the domain many more times without distortion. 
In this simulation, the solution at the final time agrees with the initial condition with a pointwise error about $10^{-9}$.

\begin{figure}[htbp] 
\centering
\includegraphics[width=0.4\textwidth]{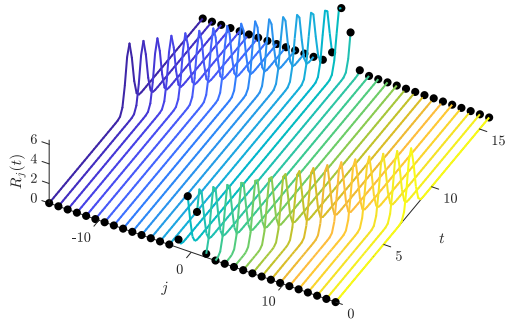}
\caption{The evolution of $R_j(t)$, with initial condition given by a solution to the traveling wave problem~\eqref{traveling}. The parameters in~\eqref{UrMiM} are $\mu=0.4$ and $\kappa\approx 45.12$, leading to a traveling wave profile in~\eqref{traveling} with $c=2$, corresponding to $n=1$ in~\eqref{omega-defn}. Initial and final conditions marked with black circles. Each curve represents the evolution of one bead $R_j$.}
\label{fig:Rjt}
\end{figure}

Equation~\eqref{KSX} relating the admissible values of $\mu$ and $\kappa$ for a given $c>1$ and $n\ge 1$ is shown in Figure~\ref{fig:kappa_mu}. 
Although Theorems~\ref{mass thm} and~\ref{stiff thm} only guarantee the existence of traveling waves when the integer constant $n$ in~\eqref{KSX} is sufficiently large, we have successfully found numerical solutions wherever we have looked in parameter space, for all values of $n$ including $n=1$.  
Figure~\ref{fig:kappa20} shows numerical solutions for $c=1.25$, $\kappa=20$ and  $n=1,2,3$. 
For $n>3$, $\rho_1$ is difficult to distinguish from $\sigma$, the traveling wave of the limiting FPUT equation. 
The $L^2$ distance between the solution and the limiting FPUT traveling wave is shown to converge like $1/n^2$. 
A similar calculation for $\mu=0.4$ and increasing $n$ is shown in Figure~\ref{fig:mu0p4}. 
Finally, the two theorems say nothing about the behavior the one-parameter families as $\mu$ and $\kappa$ are varied along the curves of fixed $n$. 
Figure~\ref{fig:mu_vs_d_nVaries} shows evidence that the solutions on these one-parameter families are $\mathcal{O}(\mu)$ close to the traveling wave of FPUT as $\mu\to0$, grow linearly in $\mu$ away from zero, and that as $n$ increases, the curves coalesce, consistently with Theorem~\ref{stiff thm}.

\begin{figure}[htbp] 
\centering
\includegraphics[width=0.5\textwidth]{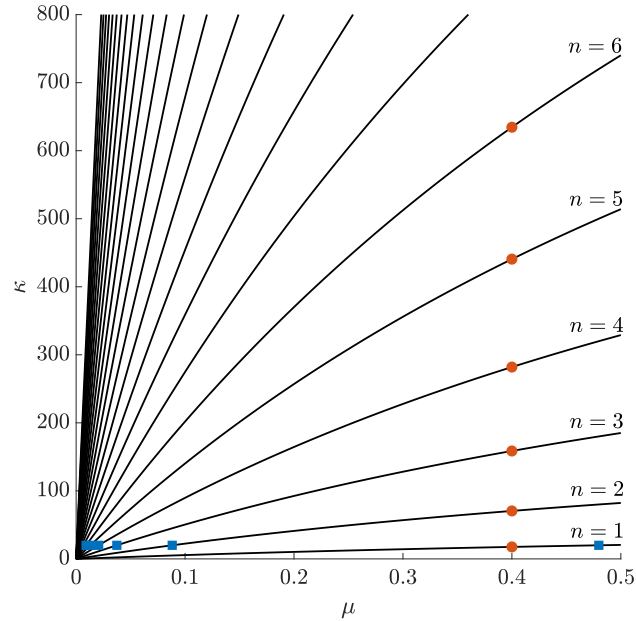} 
\caption{Curves of antiresonant $\mu$ and $\kappa$ for $c=1.25$. 
The blue squares illustrate a sequence of values with fixed $\kappa$ and decreasing mass $\mu\to 0$ as $n\to \infty$, satisfying the hypotheses of Theorem~\ref{mass thm}. 
Corresponding numerical solutions are shown in Figure~\ref{fig:kappa20}. 
The red dots mark a sequence of points with fixed $\mu=0.4$, and increasing stiffness $\kappa\to\infty$ as $n\to\infty$, satisfying the hypotheses of Theorem~\ref{stiff thm}. 
Corresponding numerical solutions are shown in Figure~\ref{fig:mu0p4}.}
\label{fig:kappa_mu}
\end{figure}

\begin{figure}[htbp] 
\centering
\includegraphics[width=0.4\textwidth]{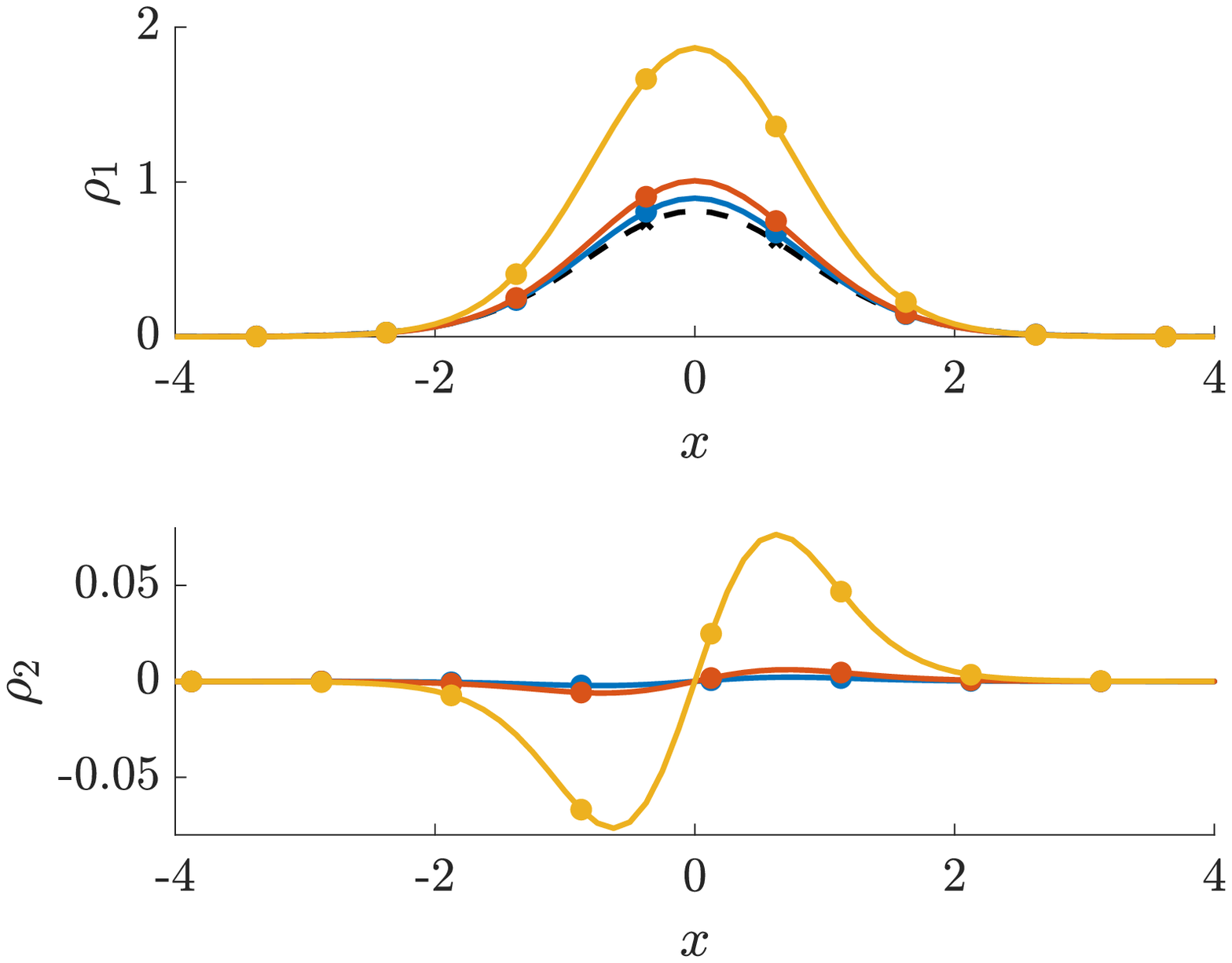} 
\includegraphics[width=0.4\textwidth]{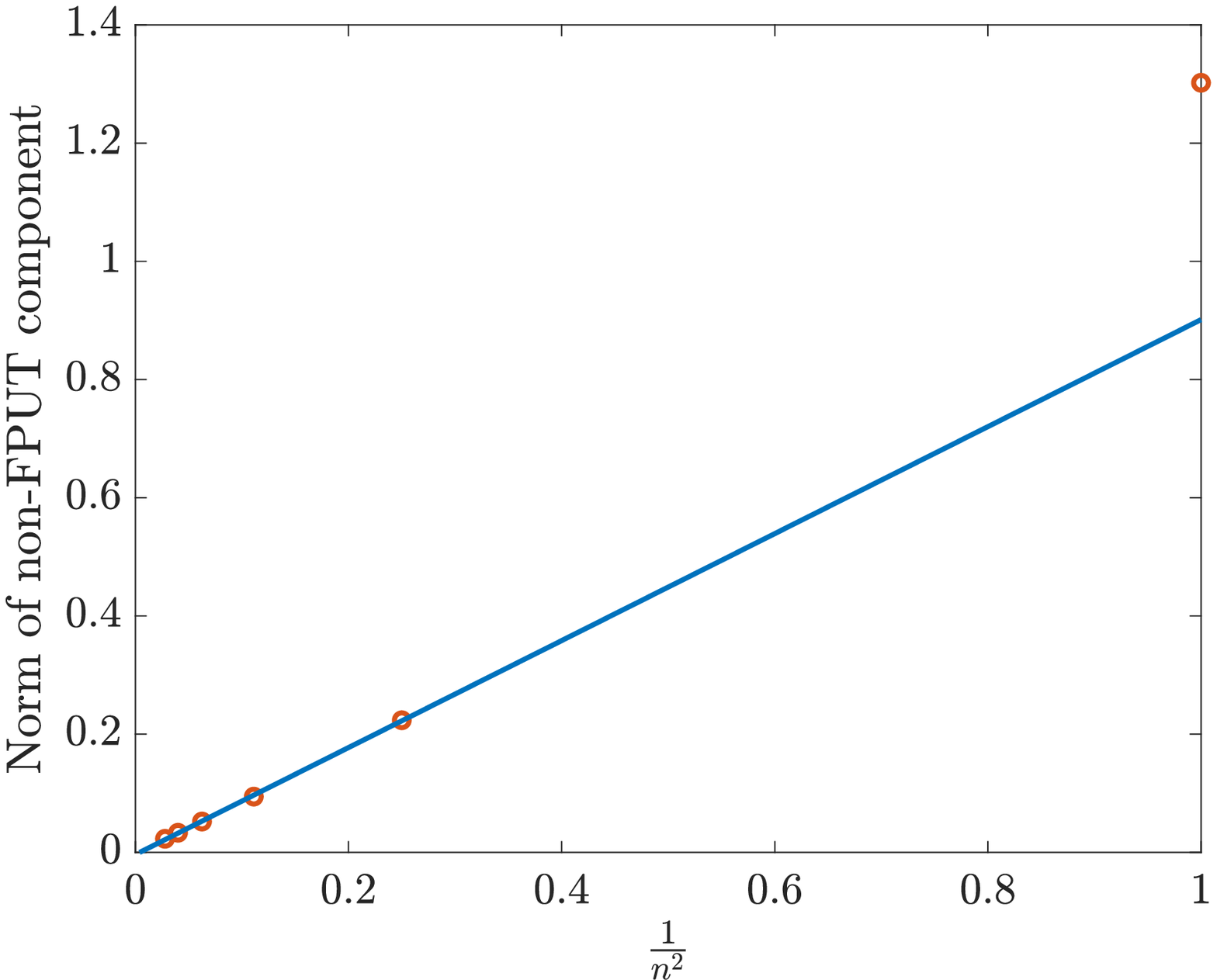} 
\caption{Left: Numerically obtained traveling waves solutions with $c=1.25$, $\kappa=20$ and $n=3,2,1$ (in blue, red and yellow), along with the limiting FPUT traveling wave $\sigma$ (black dashed). 
The discrete points forming an initial condition for $(R_j,r_j)$ are represented as dots. 
Right: A plot of $\| \sigma - \rho_1\|_{L^2} + \|\rho_2\|_{L^2}$ vs.\@ $1/n^2$, showing rapid convergence of solutions.}
\label{fig:kappa20}
\end{figure}

\begin{figure}[htbp] 
\centering
\includegraphics[width=0.4\textwidth]{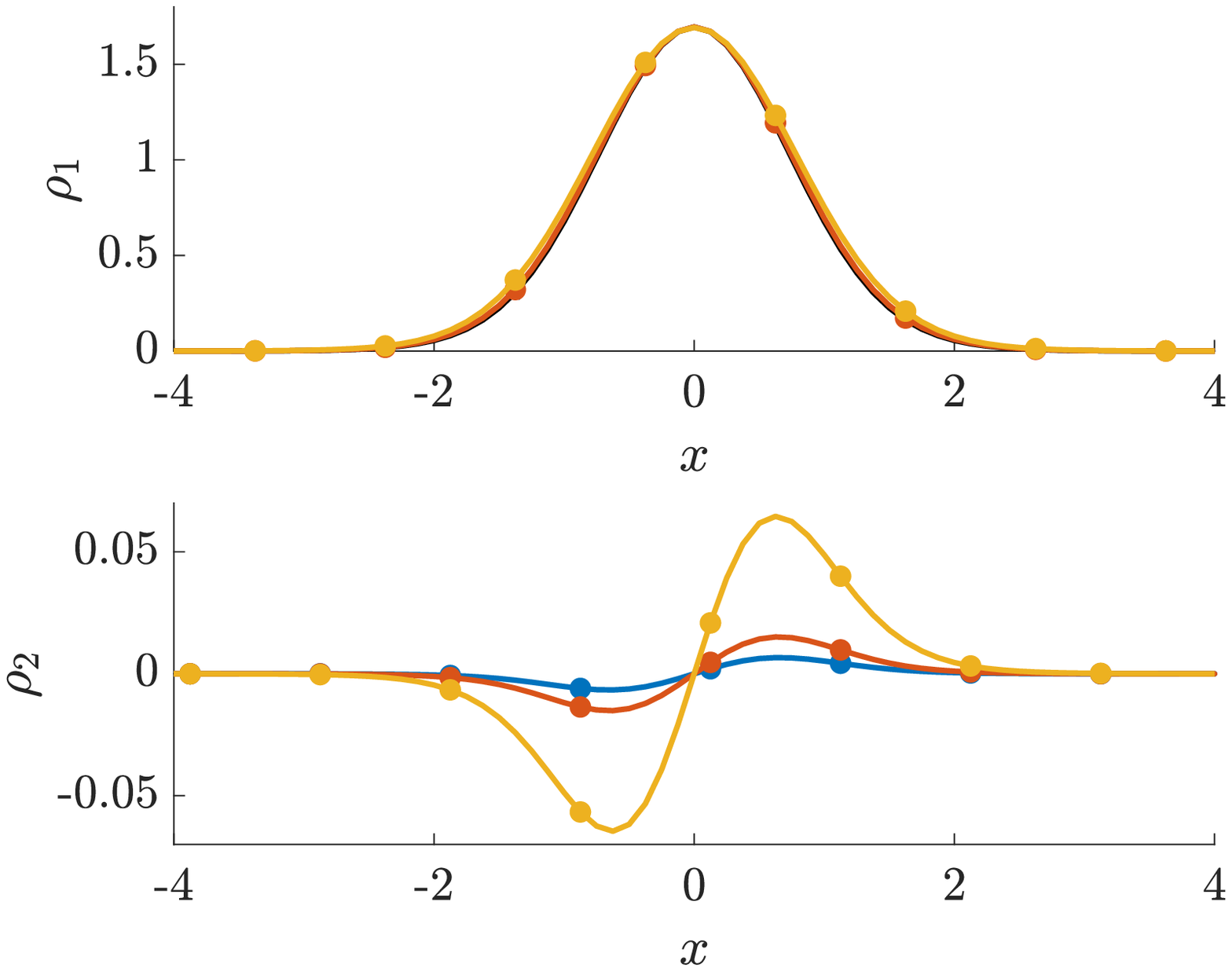} 
\includegraphics[width=0.4\textwidth]{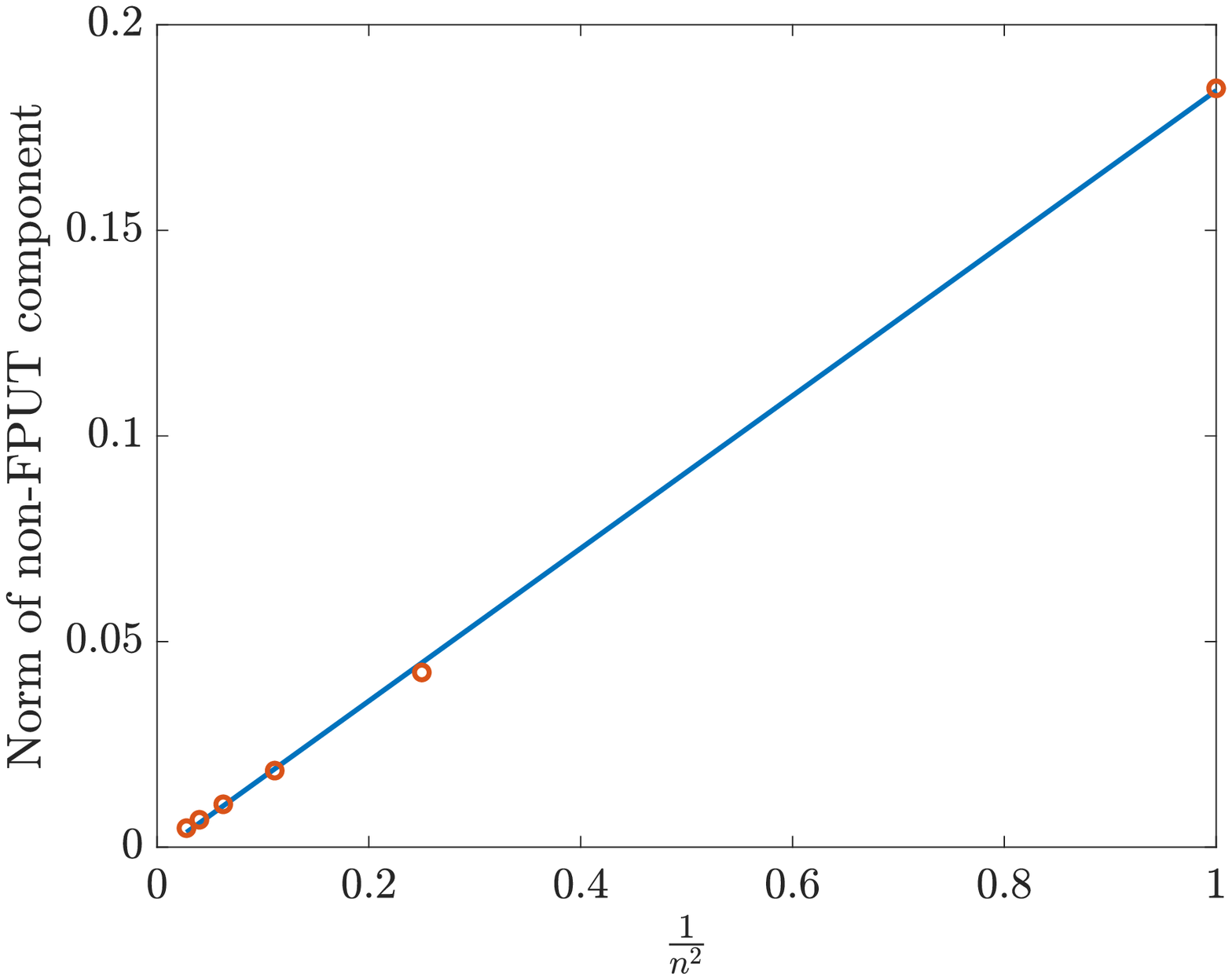} 
\caption{Left: Numerically obtained traveling waves solutions with $c=1.25$, $\mu=0.4$ and $n=3,2,1$ (in blue, red and yellow), along with the limiting FPUT traveling wave $\sigma$ (black dashed). 
Right: A plot of $\| \sigma - \rho_1\|_{L^2} + \|\rho_2\|_{L^2}$ vs.\@ $1/n^2$, showing rapid convergence of solutions.}
\label{fig:mu0p4}
\end{figure}

\begin{figure}[htbp] 
\centering
\includegraphics[width=.5\textwidth]{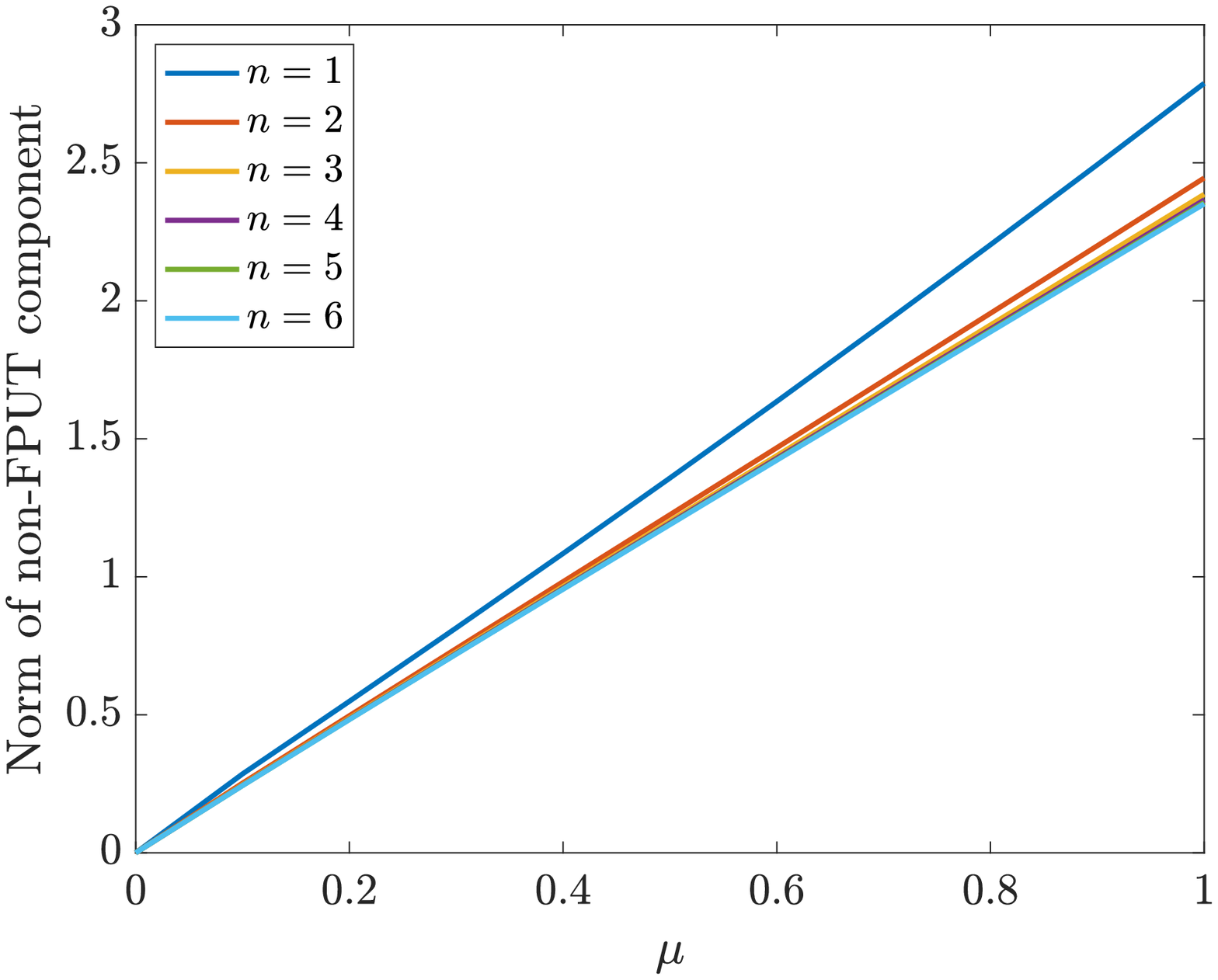}
\caption{The approach of the solutions to the FPUT traveling wave as $\mu \to 0^+$ along the curves of Figure~\ref{fig:kappa_mu}.}
\label{fig:mu_vs_d_nVaries}
\end{figure}


We have run several numerical experiments that show that if the traveling wave is unstable, then the instability must quite subtle. The simplest such experiment is simply to simulate the time-dependent dynamics using an initial condition formed by multiplying the exact traveling wave profile by $1 +\epsilon$ where $|\epsilon| \ll 1$ and may take either sign. As the traveling waves lie on a one-parameter family whose wave speed increases as a function of the norm, choosing $\epsilon>0$ results, to leading order, in a slightly faster traveling wave, and, similarly, $\epsilon<0$ leads to a slower wave. For $\epsilon = \pm 10^{-2}$, we observe traveling waves that at first glance generate no visible radiation, though
 upon zooming in to the neighborhood of the initial disturbance, i.e.\ near the site $n=0$, we find that after the traveling wave escapes from this neighborhood, it leaves behind a small disturbance that slowly disperses over time; this is likely a transient effect and not indicative of instability.
However after zooming in sufficiently close on the region immediately behind the traveling wave, we find a {\it very} small ``oscillatory wake.'' This wake is generated at each lattice site and therefore slowly drains energy (which is, of course, conserved for the system) from the solitary wave.  As such this will cause a slow, but
 substantial, attenuation of the solitary wave's amplitude.
One such simulation is shown in Figure~\ref{fig:stability}. Additional simulations using different parameters, not shown here, display an identical behavior.  By contrast, Figure~\ref{fig:fput_R32} shows the result of the analogous  experiment for the FPUT system, i.e. equation~\eqref{UrMiM} with $\mu=0$. This simulation lacks the sinusoidal oscillations in the wake of the traveling wave. It is our point of view that the oscillatory wake is evidence of a very weak and purely nonlinear instability for the antiresonant waves. Similar phenomena is known to occur in the diatomic FPUT lattice \cite{lustri, GSWW} as well as other problems which have ``embedded solitons''~\cite{benilov,tan-yang-pelinovsky}. A more detailed rigorous stability analysis is beyond the scope of this current paper, but certainly of great interest.

\begin{figure}[htbp] 
   \centering
   \includegraphics[width=0.6\textwidth]{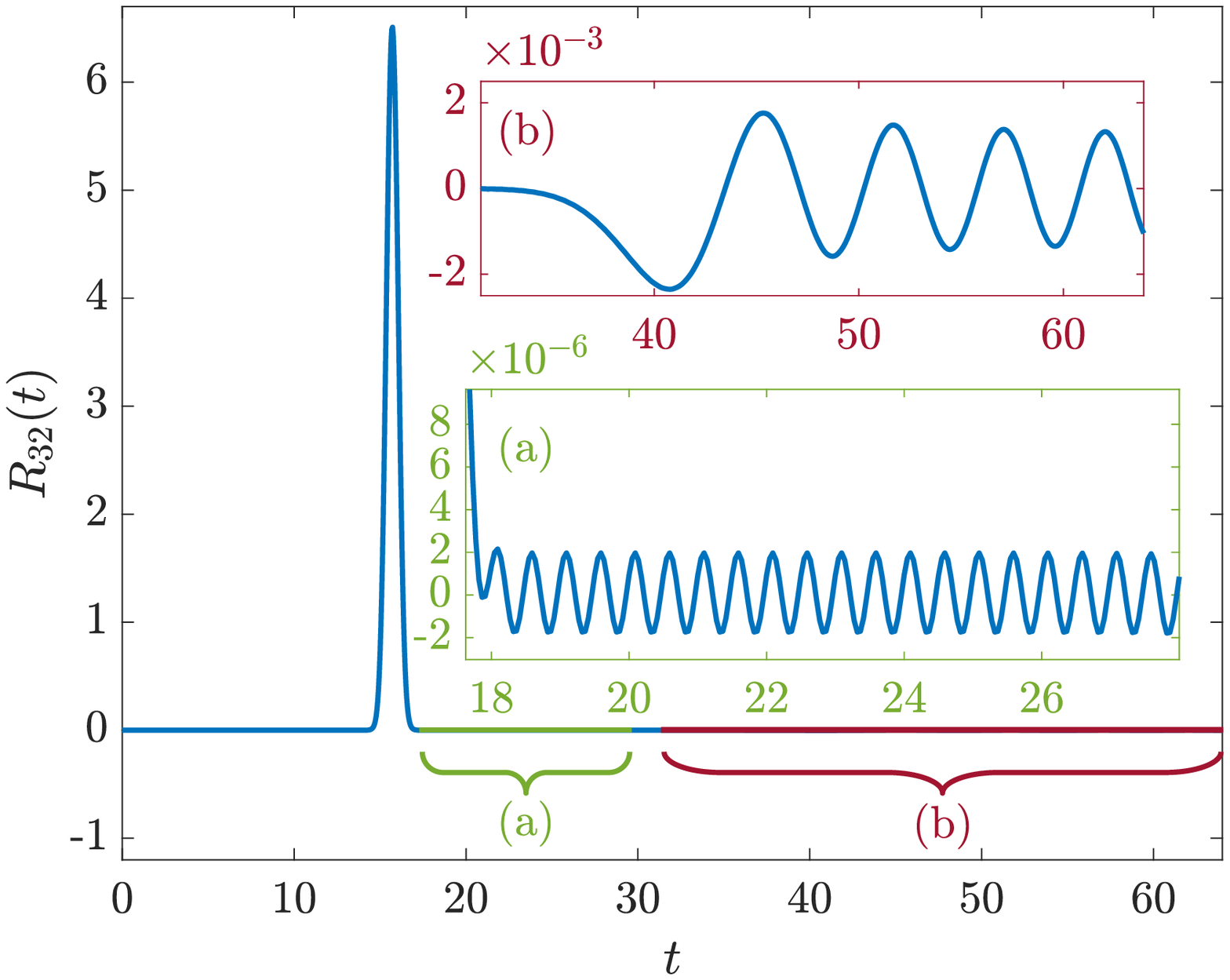} 
   \caption{Numerical results of time-dependent simulation initiated with traveling-wave initial condition multiplied by 1.01, showing the value of $R_{32}$ as a function of time; parameters as in Fig.~\ref{fig:Rjt}. In the unmagnified view, only the traveling wave is visible, which passes at about $t=16$. Inset (a) shows a zoomed view immediately after the passage of traveling wave, displaying the oscillatory wake. Inset (b) shows a zoom somewhat later, showing the arrival of the dispersing disturbance that spreads out from a neighborhood of the initial location of the wave.}
   \label{fig:stability}
\end{figure}

\begin{figure}[htbp] 
   \centering
   \includegraphics[width=0.6\textwidth]{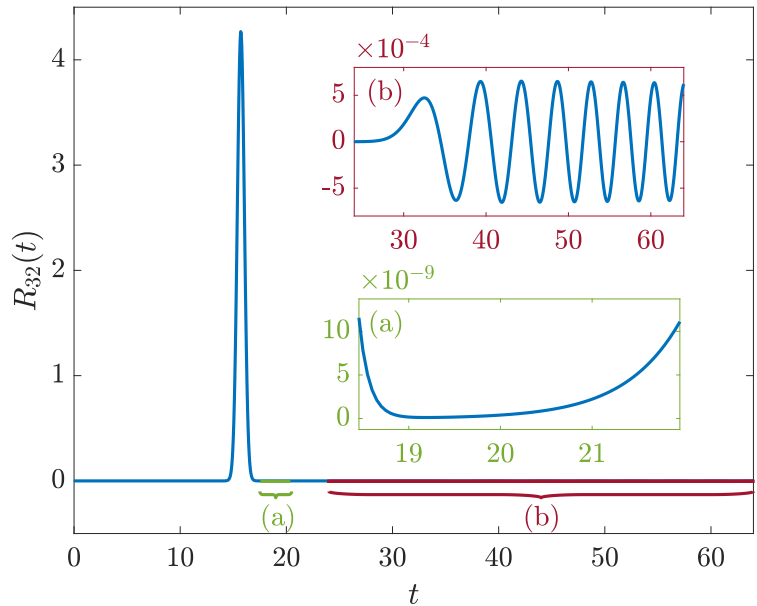} 
   \caption{Repeat of the numerical experiment summarized in Fig.~\ref{fig:stability} for the FPUT system with $\mu=0$, showing no oscillatory wake behind the propagating wave.}
   \label{fig:fput_R32}
\end{figure}

\bibliographystyle{siam}
\bibliography{mim-solitary-bib}{}
\end{document}